\documentclass{article}
\usepackage{amssymb}
\usepackage{verbatim}
\usepackage{array}
\usepackage{latexsym}
\usepackage{enumerate}
\usepackage{amsmath}
\usepackage{amsfonts}
\usepackage{amsthm}
\usepackage{color}
\usepackage[english]{babel}

\newtheorem{theorem}{Theorem}[section]
\newtheorem{proposition}[theorem]{Proposition}
\newtheorem{lemma}[theorem]{Lemma}
\newtheorem{corollary}[theorem]{Corollary}

\newtheorem{rem}[theorem]{Remark}

\def\cC{\mathcal C}

\def\cM{\mathcal M}

\def\cX{\mathcal X}
\def\cY{\mathcal Y}
\def\cZ{\mathcal Z}

\def\Aut{\mbox{\rm Aut}}
\def\K{\mathbb{K}}

\def\ord{\mbox{\rm ord}}
\def\deg{\mbox{\rm deg}}

\def\Aut{\mbox{\rm Aut}}

\def\Alt{\mbox{\rm Alt}}
\def\Sym{\mbox{\rm Sym}}

\def\gg{\mathfrak{g}}

\newcommand{\aut}{\mbox{\rm Aut}}

\newcommand{\ha}{{\textstyle\frac{1}{2}}}

\title{Large $p$-groups of automorphisms of algebraic curves in characteristic $p$}
\date{}
\author{Massimo Giulietti and G\'abor Korchm\'aros}
\begin{document}
\maketitle


    \begin{abstract}
Let $S$ be a $p$-subgroup of the $\K$-automorphism group $\aut(\cX)$ of an algebraic curve $\cX$ of genus $\gg\ge 2$ and $p$-rank $\gamma$ defined over an algebraically closed field $\mathbb{K}$ of characteristic $p\geq 3$. Nakajima \cite{nakajima1987} proved that if $\gamma \ge 2$ then $|S|\leq \textstyle\frac{p}{p-2}(\gg-1)$. If equality holds, $\cX$ is a {\emph{Nakajima extremal curve}}.
We prove that if $$|S|>\textstyle\frac{p^2}{p^2-p-1}(\gg-1)$$ then one of the following cases occurs.
 \begin{itemize}
\item[{\rm{(i)}}] $\gamma=0$ and the extension $\K(\cX)|\K(\cX)^S$  completely ramifies at a unique place, and does not ramify elsewhere.
\item[{\rm{(ii)}}] $|S|=p$, and $\cX$ is an ordinary curve of genus $\gg=p-1$.
\item[{\rm{(iii)}}] $\cX$ is an ordinary, Nakajima extremal curve,
and
$\K(\cX)$ is an unramified Galois extension of a function field of
a curve given in {\rm{(ii)}}. There are exactly $p-1$ such Galois extensions. Moreover, if some of them is an abelian extension then $S$ has maximal nilpotency class.
\end{itemize}
The full $\mathbb{K}$-automorphism group of any Nakajima extremal curve is determined, and several infinite families of Nakajima extremal curves are constructed
by using their pro-$p$ fundamental groups.
\end{abstract}

    \section{Introduction}
In the present paper, $\K$ is an algebraically closed field of
characteristic $p\geq 3$,  $\cX$ is {a} (projective, non-singular,
geometrically irreducible, algebraic) curve of genus $\gg(\cX)\geq
2$, $\mathbb K(\cX)$ is the function field
of $\cX$,
and $\aut(\cX)$ is the $\K$-automorphism group of $\cX$, and $S$
is a (non-trivial) subgroup of $\aut(\cX)$ whose order is a power of $p.$

The earliest results on the maximum size of $S$ date back to the 1970s
and have played an important role in the study of curves with large automorphism groups exceeding the classical Hurwitz bound $84(\gg(\cX)-1)$. Stichtenoth proved that if $S$ fixes a place $\mathcal P$ of $\mathbb K(\cX)$ then
\begin{equation}
\label{eq117feb2013}
|S|\leq \textstyle\frac{p}{p-1}\,\gg(\cX)
\end{equation}
unless the extension $\mathbb K(\cX)|\mathbb K(\cX)^S$ completely ramifies at $\mathcal P$, and does not ramify elsewhere; in geometric terms, $S$ fixes a point $P$ of $\cX$ and acts on $\cX\setminus\{P\}$ as a semiregular permutation group; see \cite{stichtenoth1973I} and also \cite[Theorem 11.78]{hirschfeld-korchmaros-torres2008}.
In the latter case, the Stichtenoth bound is
\begin{equation}
\label{sti16feb2013}
|S|\leq \textstyle\frac{4p}{p-1}\,\gg(\cX)^2.
\end{equation}
In his paper \cite{nakajima1987} Nakajima pointed out that the maximum size of $S$ is also related to the Hasse-Witt invariant $\gamma(\cX)$ of $\cX$. It is known
that $\gamma(\cX)$ coincides with the $p$-rank  of $\cX$ defined to be the rank of the (elementary abelian) group of the $p$-torsion points in the Jacobian variety of $\cX$; moreover,   $\gamma(\cX)\leq {\gg}(\cX)$ and when equality holds then $\cX$ is called an {\em ordinary} (or {\em general}) curve; see \cite[Section 6.7]{hirschfeld-korchmaros-torres2008}.
If $S$ fixes a point and (\ref{eq117feb2013}) fails then $\gamma(\cX)=0$; conversely, if $\gamma(\cX)=0$, then $S$ fixes a point, see  \cite[Lemma 11.129]{hirschfeld-korchmaros-torres2008}.
For $\gamma(\cX)>0$, Nakajima proved that $|S|$ divides $\gg(\cX)-1$ when $\gamma(\cX)=1$, and
$|S|\le p/(p-2) (\gamma(\cX)-1)$ otherwise; see \cite{nakajima1987} and also \cite[Theorem 11.84]{hirschfeld-korchmaros-torres2008}. Therefore,
the Nakajima bound \cite[Theorem 1]{nakajima1987} is
\begin{equation}
\label{naka16feb2013}
|S|\leq \left\{
\begin{array}{lll}
\textstyle\frac{p}{p-2}\,(\gg(\cX)-1)\quad {\mbox{for}}\quad \gamma(\cX)\ge 2,\\
\quad\quad\,\, \gg(\cX)-1\quad\,\, {\mbox{for}}\quad \gamma(\cX)=1.
\end{array}
\right.
\end{equation}
A \emph{Nakajima extremal curve} is a curve $\cX$  with $p$-rank $\gamma(\cX)\ge 2$ which attains the bound ({\ref{naka16feb2013}).

In this context, a major  issue is to determine the possibilities for $\cX$, $\gg$ and $S$ when either $|S|$ is close to the Stichtenoth bound (\ref{sti16feb2013}), or $|S|$ is close to the Nakajima bound (\ref{naka16feb2013}).

Lehr and Matignon \cite{lehr-matignon2005} investigated the case where $S$ fixes a point and were able to determine all curves $\cX$ with
\begin{equation}
\label{lehrm} |S|>\textstyle\frac{4}{(p-1)^2}\,\gg(\cX)^2,
\end{equation}
proving that (\ref{lehrm})
only occurs when the curve is birationally equivalent over $\K$ to
an Artin-Schreier curve of equation $Y^q-Y=f(X)$ such that
$f(X)=XS(X)+cX$ where $S(X)$ is an additive polynomial of $\K[X]$.
Later on, Matignon and Rocher \cite{matignon-rocher2008} showed
that the action of a $p$-subgroup of $\K$-automorphisms
$S$ satisfying
$$|S|>\textstyle\frac{4}{(p^2-1)^2}\,\gg(\cX)^2,$$ corresponds to the
\'etale cover of the affine line with Galois group
$S\cong(\mathbb{Z}/p \mathbb{Z})^n$ for $n\leq 3$. These results
have been refined by Rocher, see \cite{rocher1} and
\cite{rocher2}. The essential tools used in the above
mentioned  papers are ramification theory and some structure
theorems about finite $p$-groups.

Curves close to the Nakajima bound, and in particular Nakajima extremal curves, are investigated in this paper. Our main results are stated in the following theorems.
\begin{theorem}
\label{princ}
Let $S$ be a $p$-subgroup of the $\K$-automorphism group $\aut(\cX)$ of an algebraic curve $\cX$ of genus $\gg(\cX)\geq 2$ defined over an algebraically closed field $\mathbb{K}$ of characteristic $p\geq 3$. If
\begin{equation}
\label{hyp}
|S|>\textstyle\frac{p^2}{p^2-p-1}(\gg(\cX)-1)
\end{equation}
 then one of the following cases occurs:
 \begin{itemize}
\item[{\rm{(i)}}] $\gamma=0$ and the extension $\K(\cX)|\K(\cX)^S$  completely ramifies at a unique place, and does not ramify elsewhere.
\item[{\rm{(ii)}}] $|S|=p$, and $\cX$ is an ordinary curve of genus $\gg=p-1$.
\item[{\rm{(iii)}}] $\cX$ is an ordinary  Nakajima extremal curve, and $\K(\cX)$ is an unramified Galois extension of a function field of
a curve given in {\rm{(ii)}}. There are exactly $p-1$ such Galois extensions.
\end{itemize}
\end{theorem}
\begin{theorem}
\label{princ1} In case {\rm{(iii)}}, $S$ is generated by two elements and if one of the $p-1$ Galois extensions is abelian, then $S$ has maximal nilpotency class. If there are more than one such abelian extensions, then $\gg=p^2(p-2)+1$, $|S|=p^3$ and $S\cong UT(3,p)$ where $UT(3,p)$ is the group of all upper-triangular unipotent $3\times 3$ matrices over the field with $p$ elements.
\end{theorem}
\begin{theorem}
\label{fullaut} Let $\cX$ be an Nakajima extremal curve, and $S$ a Sylow $p$-subgroup of $\aut(\cX)$. Then either $S$ is a normal subgroup of $\aut(\cX)$ and $\aut(\cX)$ is the semidirect product of $S$ by a subgroup of a dihedral group of order $2(p-1)$, or $p=3$ and, for some subgroup $M$ of $S$ of index $3$, $M$ is a normal subgroup of $\aut(\cX)$ and $\aut(\cX)/M$ is isomorphic to a subgroup of $GL(2,3)$.
\end{theorem}

We also construct several infinite families of Nakajima extremal curves, and provide explicit equations, especially for $p=3$ and small genera.

 The analogous problem for $2$-groups of automorphisms $S$ makes sense in characteristic $p=2$ but the investigation gave rather different results, see \cite{gkp=2,gk2015}.

One may also ask how the above results may be refined when $\aut(\cX)$ is much larger than $S$. So far, this problem has been
investigated for zero $p$-rank curves $\cX$ such that $\aut(\cX)$ fixes no point of $\cX$; see  \cite{gktrans,gklondon,gur}.

The present paper is also related with the study of automorphism groups of curves in terms of quotients of fundamental groups, see \cite{bouw2001, pacheco1995,pachecostevenson2000}.

\section{Background and Preliminary Results}\label{sec2}
Let $\bar \cX$ be a non-singular model of $\K(\cX)^S$, that is,
a projective non-singular geometrically irreducible algebraic
curve with function field $\K(\cX)^S$, where $\K(\cX)^S$ consists of all elements of $\K(\cX)$
fixed by every element in $S$. Usually, $\bar \cX$ is called the
quotient curve of $\cX$ by $S$ and denoted by $\cX/S$. The
 field extension
$\K(\cX)|\K(\cX)^S$ is  Galois of degree $|S|$.

Let $\bar{P_1},\ldots,\bar{P_k}$ be the points of the quotient curve $\bar{\cX}=\cX/S$ where the cover $\cX\mapsto\bar{\cX}$ ramifies. For $1\leq i\leq k$, let $L_i$ denote the set of points of $\cX$ which lie {over} $\bar{P_i}$.
In other words, $L_1,\ldots,L_k$ are the short orbits of
$S$ on its faithful action on $\cX$. Here the orbit of $P\in \cX$
$$o(P)=\{Q\mid Q=P^g,\, g\in S\}$$ is {\em long} if $|o(P)|=|S|$, otherwise $o(P)$ is {\em short}. It may be that $S$ has no short orbits. This is the case if and only if every non-trivial element in $S$ is fixed--point-free on $\cX$. On the other hand, $S$ has a finite number of short orbits.

If $P$ is a point of $\cX$, the stabilizer $S_P$ of $P$ in $S$ is
the subgroup of $S$ consisting of all elements fixing $P$. For a
non-negative integer $i$, the $i$-th ramification group of $\cX$
at $P$ is denoted by $S_P^{(i)}$ (or $S_i(P)$ as in \cite[Chapter
IV]{serre1979})  and defined to be
$$S_P^{(i)}=\{g\mid \ord_P(g(t)-t)\geq i+1, g\in
S_P\}, $$ where $t$ is a uniformizing element (local parameter) at
$P$. Here $S_P^{(0)}=S_P^{(1)}=S_P$.

Let $\bar{\gg}$ be the genus of the quotient curve $\bar{\cX}=\cX/S$. The Hurwitz
genus formula  gives the following equation
    \begin{equation}
    \label{eq1}
2\gg-2=|S|(2\bar{\gg}-2)+\sum_{P\in \cX} d_P.
    \end{equation}
    where
\begin{equation}
\label{eq1bis}
d_P= \sum_{i\geq 0}(|S_P^{(i)}|-1).
\end{equation}

Let $\gamma$ be the $p$-rank of $\cX$, and let $\bar{\gamma}$ be the $p$-rank of the quotient curve $\bar{\cX}=\cX/S$.
The Deuring-Shafarevich formula, see \cite{sullivan1975} or \cite[Theorem 11,62]{hirschfeld-korchmaros-torres2008}, states that
\begin{equation}
    \label{eq2deuring}
\gamma-1={|S|}(\bar{\gamma}-1)+\sum_{i=1}^k (|S|-\ell_i)
    \end{equation}
where $\ell_1,\ldots,\ell_k$ are the sizes of the short orbits of $S$.
 If
 $S$ has no short orbits, that is, the Galois extension $\mathbb{K}(\cX)$ of $\mathbb{K}(\bar{\cX})$ is unramified, then
$S$ can be generated by $\bar{\gamma}$ elements by Shafarevich's theorem \cite[Theorem 2]{Shafarevich1954}, whereas the largest elementary abelian subgroup of $S$ has rank at most $\bar{\gamma}$
see  \cite[Section 4.7]{pries2011}.

The Artin-Mumford curve $\cM_c$ over a field $\mathbb{K}$ of characteristic $p>2$ is the curve birationally equivalent over $\mathbb{K}$ to the plane curve with affine equation
\begin{equation}
\label{artinmumford}
(x^p-x)(y^p-y)=c, \quad c\in \mathbb{K}^*.
\end{equation}
$\cM_c$ has genus $\textsf{g}=(p-1)^2$ and that its $\mathbb{K}$-automorphism group is isomorphic to
$(C_p\times C_p)\rtimes D_{p-1},$
where $C_p$ is a cyclic group of order $p$ and $D_{p-1}$ is a dihedral group of order $2(p-1)$; see see \cite{maddenevalentini1982}, and \cite[Theorem 11.93]{hirschfeld-korchmaros-torres2008}.

\begin{proposition}
\label{igusa1} Let $\cY$ be a curve of genus $p-1$ and positive $p$-rank such that $p$ divides $\aut(\cY)$. If  $G$ is a subgroup of $\aut(\cY)$ containing a subgroup $T$ of order $p$, then either
$T$ is a normal subgroup and $G=T\rtimes H$ with $H$ a subgroup of a dihedral group of order $2(p-1)$, or $p=3$ and $\cY$ is a non-singular model of the plane curve with affine equation
\begin{equation}
\label{eqago1}
Y^3-Y=-X+\frac{1}{X},
\end{equation}
and $\aut(\cY)\cong GL(2,3)$.
\end{proposition}
\begin{proof} Let $T$ be a subgroup of $\aut(\cY)$ of order $p$. The Hurwitz  genus formula applied to $T$ yields that the number $\lambda$ of fixed points of $T$ on $\cY$ is positive.
From the Deuring-Shafarevich formula applied to $T$,
$p-2\ge \gamma-1=p(\bar{\gamma}-1)+\lambda(p-1)$ whence $\bar{\gamma}=0$ and $\lambda=2$. Now, from the Hurwitz genus formula  applied to $T$, $2(p-2)\geq 2p(\bar{\gg}-1)+4(p-1)$ which yields  $\bar{\gg}=0$. Therefore, $T$ is a normal subgroup $G$ with four exceptions by a result of Madan and Valentini \cite{maddenevalentini1982}; see also \cite[Theorem 11.93]{hirschfeld-korchmaros-torres2008}. One exception occurs for $p=3$ when $\cY$ is a non-singular model of a plane curve $\cC$ of affine equation $X(X-1)(Y^3-Y)=\alpha$ with $\alpha^2=2$, equivalently (\ref{eqago1}), and $G$ is isomorphic to a subgroup of $GL(2,3)$. This shows that Proposition \ref{igusa1} holds in this case. Two of the other three exceptions have zero $p$-rank, while the fourth is the Artin-Mumford curve of genus $(p-1)^2$. Therefore, they cannot actually occur in our case.

We may assume that $T$ is a normal subgroup of $G$. By the Nakajima bound
\eqref{naka16feb2013} applied to $\cY$, $T$ is a Sylow $p$-subgroup of $\aut(\cY)$. Therefore,
$G=T\rtimes H$ with $H$ of order prime to $p$. Therefore, $H$ can be viewed as a subgroup of the rational curve fixing two points. Hence,
$H$ is a subgroup of a dihedral group of order $2(p-1)$.
\end{proof}
\begin{rem}
\label{rem1ago2014} {\em{Apart from the exceptional case $p=3$ and $a=-1$, a non-singular model of the plane curve $\cC_a$ with affine equation
\begin{equation}
\label{eqago11}
Y^p-Y=aX+\frac{1}{X}, \quad a\in \mathbb{K}^*
\end{equation}
is a general (hyperelliptic) curve of genus $p-1$ which provides an example for the curve $\cY$ in Proposition \ref{igusa1} with an elementary abelian group $H$ of order $4$, so that
$G=\langle h \rangle \times D_p$ where $h$ is the hyperelliptic involution and $D_p$ is the dihedral group of order $2p$. If $\mathbb{K}$ is the algebraic closure of the finite field $\mathbb{F}_p$, then $G=\aut(\cY)$ by a result due to van der Geer and der Vlugt \cite{vgvv}. As far as we know, no curve $\cY$ with a larger subgroup $H$ is available in the literature.}}
\end{rem}

\begin{rem}
\label{igusa1bis}
{\em{Let $p=3$. The plane curve $\cC_a$ in Remark \ref{rem1ago2014} has also an affine equation of type 
\begin{equation}
\label{eq2} Y^2=cX^6+X^4+X^2+1
\end{equation}
with some $c\in \mathbb{K}^*$, and provides a further plane model of the curve $\cY$ defined in Proposition \ref{igusa1},  see \cite[Section 8]{igu}, and \cite[Section 1]{car}; see also \cite[Lemma 1]{sb}, and \cite{carq}. In particular, $\aut(\cY)$ is a dihedral group of order $12$, apart from the exceptional case (\ref{eqago1}) occurring here for $c=1$. It is an open problem to decide whether an analog result may hold for $p\geq 5$.
}}
\end{rem}
{}From Galois theory we use results on the pro-$p$ fundamental group $\pi_1^p(\bar{\cX})$ of an algebraic curve $\bar{\cX}$ with $p$-rank $\bar{\gamma}$ greater than $1$; see \cite{pries2011} and
\cite{Shafarevich1954}. The (finite, Galois)  $p$-extensions of $\K(\bar{\cX})$ are taken in a given separable algebraic closure of  $\K(\bar{\cX})$.
\begin{proposition}
\label{sharata} The pro-$p$ fundamental group $\pi_1^p(\bar{\cX})$ is a free group $\Gamma$ generated by $\bar{\gamma}$ generators. The unramified $p$-extensions of $\K(\bar{\cX})$ are in one-to-one correspondence with the normal subgroups of $\pi_1^p(\bar{\cX})$ whose indices are powers of $p$. Moreover, if an unramified $p$-extension $F$ corresponds to the normal subgroup $N$ then the Galois group ${\rm{Gal}}(F|\K(\bar{\cX}))$ is isomorphic to the factor  group $\Gamma/N$. If two unramified $p$-extensions $F$ and $F_1$ correspond to $N$ and $N_1$, respectively, then $F\supseteq F_1$ implies $N \subseteq N_1$ and conversely.
\end{proposition}
\begin{proposition}
\label{shafart} Let $G$ be a finite $p$-group. If $d(G)$ is the minimum size of the generator sets of $G$, and $\alpha(G)$ is the order of the automorphism group of $G$, then the following statements hold.
\begin{itemize}
\item[\rm(i)] There exists an unramified $p$-extension of $\mathbb{K}(\bar{\cX})$ with Galois group isomorphic to $G$ if and only if $d(G)\leq \gamma$.
\item[\rm(ii)] If $d(G)\leq \gamma$ then the number of different unramified $p$-extensions of $\mathbb{K}(\bar{\cX})$ with Galois group isomorphic to $G$ is equal to
\begin{equation}
\label{frbound} \frac{p^{{\gamma}(n-d(G))}(p^{\gamma}-1)(p^{\gamma}-p)\cdots(p^{\gamma}-p^{d(G)-1})}{\alpha(G)}.
\end{equation}
\end{itemize}
\end{proposition}
{}From group theory we use the following results; see  \cite[Theorem 12.2.2]{hall1976} and \cite[Chapter III, 3.19 Satz]{huppertI1967}.
\begin{proposition} [Burnside-Hall bound]
\label{BH} Let $G$ be a $p$-group of order $p^n$. If $d(G)$ is the minimum size of the generator sets of $G$ and $\alpha(G)$ is the order of the automorphism group of $G$, then
$\alpha(G)$ divides
\begin{equation}
\label{eq25agostoA}
p^{d(G)(n-d(G))}\,(p^{d(G)}-1)(p^{d(G)}-p)\cdots (p^{d(G)}-p^{d(G)-1}).
\end{equation}
In particular, the order of a Sylow $p$-subgroup of the automorphism group of $G$ divides
\begin{equation}
\label{eq16may2014}
p^{d(G)(n-d(G))+\frac{1}{2}d(G)(d(G)-1)}.
\end{equation}
\end{proposition}
Comparison of the above two propositions, especially (\ref{eq16may2014}) with (\ref{frbound}), gives the following result.
\begin{corollary}
\label{25agosto2014} Let $G$ be any finite $p$-group. If the minimum size of the generator sets of $G$ is equal to the Hasse-Witt invariant of $\bar{\cX}$ then the number of unramified $p$-extensions of $\mathbb{K}(\bar{\cX})$ with Galois group isomorphic to $G$ is not divisible by $p$.
\end{corollary}

\begin{rem}
\label{rem15may2014} \em{Well known groups $G$ whose automorphism groups attain  (\ref{eq25agostoA}) are the direct product of $d(G)$ copies of the
cyclic group of order $p^N$ where $N$ is any positive integer. Furthermore, the Sylow $p$-subgroup of the special linear group $SL(p,p)$
is isomorphic to the group $U(p,p)$ of all non-degenerate upper unitriangular ($p\times p$)-matrices over $\mathbb{F}_p$ and the minimum size of the generator sets of $U(p,p)$ is equal $p-1$. Therefore, Corollary \ref{25agosto2014} applies to any curve $\bar{\cX}$ with Hasse-Witt invariant equal to $p-1$. Using the database of GAP,  more such examples can be obtained for smaller $p$.}
\end{rem}
{}From Projective geometry, the following known result is used.
\begin{lemma}\label{flag}
In the $r$-dimensional projective space  $PG(r,\mathbb{K})$ over an algebraically closed field $\mathbb{K}$ of characteristic $p$, let $S$ be a finite $p$-subgroup of $PGL(r+1,\mathbb K)$. If $r\geq 2$ then $S$ preserves a
flag $$\Pi_0\subset \Pi_1 \subset \ldots \subset \Pi_{r-1}$$
where $\Pi_i$ is an $i$-dimensional projective
subspace
of $PG(r,\mathbb{K})$.
\end{lemma}

\section{Proof of Theorem \ref{princ}.}
\label{princip}
In this section, $\cX$ stands for a curve which satisfies the hypotheses of Theorem \ref{princ}.

{}From \cite[Lemma 11.129]{hirschfeld-korchmaros-torres2008}, we have the following result.
\begin{lemma}
\label{lem17Afeb2014} If $\gamma=0$ then {\rm{(i)}} of Theorem \ref{princ} holds.
\end{lemma}
Moreover, (\ref{naka16feb2013}) rules out the possibility that case $\gamma=1$ occurs  in Theorem \ref{princ}. Therefore,
\begin{equation}
\label{vecchioII} \gamma\geq 2.
\end{equation}
\begin{lemma}
\label{lem17feb2014} If $S$ fixes a point of $\cX$ then {\rm{(ii)}} of Theorem \ref{princ} holds.
\end{lemma}
\begin{proof} Comparison of  (\ref{hyp}) with (\ref{eq117feb2013}) gives
$$|S|<p^2+\textstyle\frac{p(p-1)}{p-2}.$$
Since the right hand side is smaller than $p^3$, either $|S|=p$ or $|S|=p^2$ holds. In the latter case, (\ref{hyp}) yields $g<p(p-1)$ but this contradicts (\ref{eq117feb2013}). If $|S|=p$, then
(\ref{hyp}) reads $(p^2-p-1)>p(\gg(\cX)-1)$ while (\ref{eq117feb2013}) yields $\gg(\cX)-1\geq p-2$. Therefore $\gg(\cX)-1$ is an integer in the interval $[p-2,(p^2-p-1)/p)$ whose length is smaller than $2$. This is only possible when either $\gg(\cX)-1=p-2$ or $\gg(\cX)-1=p-1$.  Comparison with (\ref{hyp}) rules out the latter case. So $\gg(\cX)=p-1$. From Nakajima's bound $|S|\le p/(p-2) (\gamma(\cX)-1)$, we have $\gamma(\cX)\geq p-1$. Therefore $\gamma(\cX)=\gg(\cX)=p-1$.
\end{proof}
{}From now on we assume that neither (i) or (ii) of Theorem \ref{princ} hold for $\cX$. In particular,
\begin{equation}
\label{eqcaseii}
|S|\geq p^2.
\end{equation}
\begin{proposition}
\label{propap=3} $\cX$ is an ordinary Nakajima extremal curve. Moreover,
$S$ has exactly two short orbits on $\cX$, both of length $\textstyle\frac{1}{p}|S|$, and the identity is the unique element in $S$ fixing every point of the short orbits.
\end{proposition}
\begin{proof} Let $\gg=\gg(\cX)$ and $\gamma=\gamma(\cX)$ where $\gamma\geq 2$ by (\ref{vecchioII}). Let $\bar{\gamma}$ be the $p$-rank of the quotient curve $\bar{\cX}=\cX/S$.
{}From (\ref{eq2deuring}),
\begin{equation}
\label{eq9feb2013}
\gamma -1=\bar{\gamma}|S|-|S|+\sum_{i=1}^k(|S|-\ell_i)=(\bar{\gamma}+k-1)|S|-\sum_{i=1}^k\ell_i\ge (\bar{\gamma}+\textstyle\frac{p-1}{p}k-1)|S|,
\end{equation}
where $\ell_1,\ldots,\ell_k$ are the sizes of the short orbits of
$S$.

If no such short orbits exist, then $\gamma-1=|S|(\bar{\gamma}-1)$  whence $\bar{\gamma}>1$ by $\gamma\geq 2$.
Therefore, $|S|{\le }\gamma-1\leq \gg-1$ contradicting (\ref{hyp}).

Hence $k\geq 1$, and if $\bar{\gamma}\geq 1$ then (\ref{eq9feb2013}) yields that $|S|\leq \textstyle\frac{p}{p-1}(\gamma-1)$ contradicting (\ref{hyp}).
So, $\bar{\gamma}=0$, and (\ref{eq9feb2013}) together with (\ref{hyp}) imply that

$$k<\textstyle\frac{2p^2-p-1}{p^2-p}=2+\textstyle\frac{1}{p}$$ whence $1\le k \le 2$. The case $k=1$ cannot actually occur by (\ref{eq9feb2013}).

Therefore, $\bar{\gamma}=0$ and $k=2$. Let $\Omega_1$ and $\Omega_2$ be the short orbits of $S$, and let $\ell_i=|\Omega_i|$ for $i=1,2$. Then (\ref{eq9feb2013}) reads
\begin{equation}
\label{eq9afeb2013}
\gamma-1=|S|-(\ell_1+\ell_2).
\end{equation}
Also, $\ell_1+\ell_2<|S|$. Write $|S|=p^h,\ell_1=p^m,\ell_2=p^r$ with $h>m\geq r$.
Here $r>0$ by Lemma \ref{lem17feb2014}. From (\ref{hyp}) and (\ref{eq9afeb2013}),
$$\textstyle\frac{p^2}{p^2-p-1}(p^m+p^r)>p^h(\textstyle\frac{p^2}{p^2-p-1}-1),$$
whence $p^{2+m-h}+p^{2+r-h}>p+1.$ Since $m\geq r$, this yields $m=h-1$. Hence, $p^{2+r-h}>1$, and $h-1=m\geq r\geq h-1$. Therefore,
$$\ell_1=\ell_2=\textstyle\frac{|S|}{p}.$$
Let $\bar{\gg}$ be the genus of the quotient curve $\bar{\cX}=\cX/S$.
The Hurwitz genus formula applied to $S$ gives
\begin{equation}
\label{eq19feb}
2\gg-2= |S|(2\bar \gg-2)+
\textstyle\frac{p-1}{p}|S|(4+k_1+k_2)
\end{equation}
where, for a point $P_i\in \Omega_i$,  $k_i$ is the smallest non-negative integer such that $|S_{P_i}^{(2+k_i)}|=1$. Suppose on the contrary that $\cX$ is not an ordinary curve. Then $k_1+k_2\geq 1$. From (\ref{eq19feb}), $$
2g-2\ge -2|S|+5|S|\textstyle\frac{p-1}{p}=|S|(\textstyle\frac{3p-5}{p}).
$$
 Comparing this with (\ref{hyp}) yields
 $$
 \frac{2p}{3p-5}\ge \frac{|S|}{g-1}\ge \frac{p^2}{p^2-p-1},
 $$ a contradiction.

Assume that a non-trivial element $s\in S$ of order $p$ fixes $\Omega_1\cup \Omega_2$ pointwise. From the Deuring-Shafarevich formula applied to $\langle s \rangle$,
$$\textstyle\frac{p-2}{p}|S|\geq -p + 2\textstyle\frac{|S|}{p}(p-1),$$
which is only possible for $|S|=p$.
\end{proof}
We stress that the first claim of Proposition \ref{propap=3} means that
\begin{equation}
\label{eq18feb2014} \gg-1=\gamma-1=\textstyle\frac{p-2}{p}|S|,
\end{equation}
and hence $\cX$ is a Nakajima extremal curve.
\begin{proposition}
\label{prop6febb2013} $\cX$ is not hyperelliptic.
\end{proposition}
\begin{proof} Since the length of any $S$-orbit in $\cX$ is divisible by $p$, the number of distinct Weierstrass points of $\cX$ is also divisible by $p$. On the other hand, a hyperelliptic curve of genus $\gg$ defined over a field of zero or odd characteristic has as many as $2\gg+2$ Weierstrass points, see \cite[Theorem 7.103]{hirschfeld-korchmaros-torres2008}. Therefore, if $\cX$ were hyperelliptic, both numbers $\gg+1$ and $\gg-1=\frac{p-2}{p}|S|$ would be divisible by $p$, a contradiction with $|S|\geq p^2$.
\end{proof}

{}From the rest of the paper, we keep up our notation; in particular $\Omega_1$ and $\Omega_2$ denote the short orbits of $S$ on $\cX$. By the second claim of Proposition \ref{propap=3}, the following hold.
\begin{lemma}
\label{eq130apr2013}  For every point $P\in \Omega_1 \cup \Omega_2$, the stabilizer $S_P$ of $P$ has order $p$.
\end{lemma}
\begin{proposition}
\label{propbp=3} If $S$ is  abelian then $|S|=p^2$ and $S$ is elementary abelian.
\end{proposition}
\begin{proof} Choose a point $P\in \Omega_1$. From Lemma (\ref{eq130apr2013}), $|S_P|=p$. Since $S$ is abelian $S_P$ fixes every point in $\Omega_1$. Let $\gamma^*$ be the $p$-rank of the quotient curve $\cX/S_P$. The Deuring-Shafarevich formula  applied to $S_P$ together with (\ref{eq18feb2014}) give
$$\textstyle\frac{p-2}{p}|S|=\gamma-1\geq -p+\textstyle\frac{p-1}{p}|S|$$
whence $|S|\leq p^2$. Then $|S|=p^2$ by (\ref{eqcaseii}). Assume on the contrary that $S$ is cyclic. For a point $Q\in \Omega_2$ the stabilizer $S_Q$ is a subgroup of $S$ of order $p$. Since $S$ is cyclic, it has only one subgroup of order $p$. Therefore $S_P=S_Q$, and
$$\textstyle\frac{p-2}{p}|S|=\gamma-1\geq -p+2\textstyle\frac{p-1}{p}|S|$$
which implies $|S|\leq p$, a contradiction.
\end{proof}
\begin{proposition}
\label{propcp=3} Let $N$ be a non-trivial normal subgroup of $S.$ Then either $N$ is semiregular on $\cX$, or $N$ has order $\frac{|S|}{p}$ and there is point $P\in \Omega_1\cup \Omega_2$ such that $S=N\rtimes S_P$.
\end{proposition}
\begin{proof} The assertion trivially holds for $|S|=p^2$ with $S=N\times S_P$. Assume that some non-trivial element in $N$ fixes point $P$. From the Hurwitz genus formula applied to $N$, we have
$\textstyle\frac{p-2}{p}|S|>|N|(\bar{\gg}-1)$ where $\bar{\gg}$ is the genus of the quotient curve $\bar{\cX}=\cX/N$. Let $\bar{S}$ be the automorphism group of $\bar{\cX}$ induced by $S$. Then $|\bar{S}||N|=|S|$ and hence
$\textstyle\frac{p-2}{p}|\bar{S}|>\bar{\gg}-1.$ If $\bar{\gg}\geq 2$, Nakajima's bound (\ref{naka16feb2013}) applied to $\bar{\cX}$ implies that $\bar{\gamma}=0$. From \cite[Lemma 11.129]{hirschfeld-korchmaros-torres2008}, $\bar{S}$ fixes a point $\bar{Q}$ in $\bar{\cX}$. Then the orbit $\mathcal{O}$ of $N$ consisting of all points of $\cX$ lying over $\bar{Q}$ is also an orbit of $S$. Since $\Omega_1$ and $\Omega_2$ are the only short orbits of $S$, this yields that $\mathcal{O}$ coincides with one of them, say $\Omega_1$.
Therefore, $|N|=\textstyle\frac{1}{p}|S|$. The stabilizer $\varepsilon$ of a point $R\in\Omega_2$ on $S$ has order $p$ and $\varepsilon\not\in N$. Therefore $S=N\rtimes \langle\varepsilon\rangle$. This argument also works when $\bar{\gg}\le 1$ and $\bar{\gamma}=0$.
We are left with the case $\bar{\gg}=\bar{\gamma}=1$. Let $\mathcal{O}_1,\ldots,\mathcal{O}_m$ be the short orbits of $N$. Since the stabilizer $N_Q$ of any point $Q\in \mathcal{O}_i$ has order $p$, the Deuring-Shafarevich formula
 applied to $N$ together with (\ref{eq18feb2014}) give
$$\textstyle\frac{p-2}{p}|S|=\textstyle\frac{p-1}{p}|N|m$$
whence $|S|=\textstyle\frac{p-1}{p-2}|N|m$. But this is impossible as both $|S|$ and $|N|$ are powers of $p$.
\end{proof}
\begin{proposition}
\label{propgp=3} The center $Z(S)$ of $S$ is semiregular on $\cX$.
\end{proposition}
\begin{proof} Since $Z(S)$ is a normal subgroup of $S$, Proposition \ref{propcp=3} applies to $Z(S)$. The case $S=Z(S) \rtimes S_P$ cannot actually occur since this semidirect product would be direct and $S$ would be abelian contradicting Proposition \ref{propbp=3}.
\end{proof}
\begin{proposition}
\label{a30ag2013} Let $N$ be a non-trivial normal subgroup of $S$ such that $|N|\le \frac{1}{p^2}|S|$. Then the quotient curve $\bar{\cX}=\cX/N$ with $\bar{S}=S/N$ and $\gg(\bar{\cX})-1=(\gg-1)/|N|$ satisfies the hypotheses of Theorem \ref{princ} but does not have the property given in either {\rm(i)} or {\rm(ii)} of Theorem \ref{princ}. In particular, if $\cX$ is a Nakajima extremal curve then $\bar{\cX}$ is also a Nakajima extremal curve.
\end{proposition}
\begin{proof} By Proposition \ref{propcp=3}, the extension $\K(\cX)|\K(\bar{\cX})$ is an unramified $p$-extension with Galois group $N$. Therefore, the Hurwitz formula applied to $N$ gives that
$\gg-1=|N|(\gg(\bar \cX)-1)$. In Theorem \ref{princ} referred to $\bar{\cX}$ and $\bar{S}$, case (i) is impossible by $\bar{\gamma}\neq 0$, while case (ii) cannot occur since $|\bar{S}|>p$.
\end{proof}
Since the center of any $p$-group is non-trivial, a straightforward inductive argument on $|S|$ depending on Proposition \ref{a30ag2013} gives the following result.
\begin{proposition}
\label{pro12dic} If there exists a curve $\cX$ which satisfies the hypothesis of Theorem \ref{princ} for $|S|=p^k$ but does not have the properties {\rm(i)} and {\rm(ii)}, then for any $1<j <k$ the curve $\cX$ has a quotient curve $\bar{\cX}$ which satisfies the hypothesis of Theorem \ref{princ} for $|\bar{S}|=p^j$ but has none of the properties {\rm(i)} and {\rm(ii)}. 
\end{proposition}
A corollary of Propositions \ref{propcp=3} and \ref{a30ag2013} is stated in the following proposition.
\begin{proposition}
\label{propdp=3} Let $N$ be a non-trivial normal subgroup of $S$. If the factor group $S/N$ is abelian then either $|N|=\textstyle\frac{1}{p}|S|$ or $|N|=\textstyle\frac{1}{p^2}|S|$, and in the latter case, $S/N$ is an elementary abelian group.
\end{proposition}
Proposition \ref{propdp=3} together with classical results from Group theory give some useful results on $S$.
\begin{proposition}
\label{30ag2013} Let $\Phi(S)$ and $S'$ be the Frattini subgroup and the commutator subgroup of $S$, respectively. Then the following hold.
\begin{itemize}
\item[\rm(i)] $\Phi(S)=S'$.
\item[\rm(ii)] $|\Phi(S)|=\textstyle\frac{1}{p^2}|S|$.
\item[\rm(iii)] $S$ contains exactly $p+1$ maximal subgroups, each being a normal subgroup of $S$ of index $p$.
\item[\rm(iv)] Exactly two of the $p+1$ maximal subgroups of $S$ are not semiregular on $\cX$.
\item[\rm(v)] Two elements of $S$ of order $p$, one fixing a point in $\Omega_1$ and the other in $\Omega_2$, always generate $S$.
\end{itemize}
\end{proposition}
\begin{proof} From Proposition \ref{propdp=3}, either $|\Phi(S)|=\frac{1}{p}|S|$, or  $|\Phi(S)|=\frac{1}{p^2}|S|$. In the former case, $S$ is cyclic by \cite[Hilfssatz 7.1.b]{huppertI1967}  but this contradicts Proposition \ref{propbp=3}. Therefore, (ii) holds.  Since $S/\Phi(S)$ is (elementary) abelian, $\Phi(S)$ contains $S'$. Hence, Proposition \ref{propdp=3} yields  (i).
Let $\varphi$ be the natural homomorphism $S\mapsto S/\Phi(S)$. Since every maximal subgroup of $S$ contains $\Phi(S)$, there is a one-to-one correspondence between the maximal subgroups of $S$ and the subgroups of $S/\Phi(S)$. By (ii), $S/\Phi(S)$ is an elementary abelian group of order $p^2$ which have exactly $p+1$ proper subgroups. Therefore there are exactly $p+1$ maximal subgroups in $S$. Also, the subgroups of $S/\Phi(S)$ are normal, and hence each of the $p+1$ maximal subgroups of $S$ is normal, as well. Furthermore, the $p+1$ maximal subgroups of $S/\Phi(S)$ partition the set of non-trivial elements of $S/\Phi(S)$. Hence every element of $S\setminus\Phi(S)$ belongs to exactly one of the $p+1$ maximal subgroups of $S$. Take a point $P\in \Omega_1$, and let $M_1$ be the maximal subgroup of $S$ containing $S_P$. Since $M$ is a normal subgroup of $S$ and $\Omega_1$ is an $S$-orbit, this yields that $M$ contains $S_Q$ for every $Q\in \Omega_1$.  Repeating the above argument for a point in $\Omega_2$ shows that a maximal normal subgroup contains the stabilizer of each point in $\Omega_2$. From the last claim of Proposition \ref{propap=3}, these two maximal subgroups are distinct. Therefore, the remaining $p-1$ maximal subgroups are  semiregular on $\cX$.

Finally, (i) together with the Burnside fundamental theorem, \cite[Chapter III, Satz 3.15]{huppertI1967} imply that $S$ can be generated by two elements. Here any two non trivial elements from different maximal subgroups of $S$ generate $S$. Since some element $g_1$ of order $p$ fixes a point $\Omega_1$, and the same holds for some element $g_2$ fixing a point of $\Omega_2$ where $g_1,g_2$ are in two distinct maximal subgroups of $S$, it turns out that $S=\langle g_1,g_2\rangle$.
\end{proof}
{}From now on, the following notation is used: For $i=1,2,$ $M_i$ denotes the maximal normal subgroup of $S$ containing the stabilizer of a point of $\Omega_i$ while $M_3,\ldots, M_{p+1}$ stand for the semiregular maximal subgroups of $S$, respectively.
\begin{proposition}
\label{b30ago2013} Every normal subgroup of $S$ whose order is at most $\frac{1}{p^2}|S|$ is contained in $\Phi(S)$.
\end{proposition}
\begin{proof} Let $N$ be a normal subgroup of $S$. {}From \cite[Chapter III, Hilfssatz 3.4.a]{huppertI1967}, $\Phi(S)N/N$ is a subgroup of $\Phi(S/N)$. From Propositions \ref{a30ag2013} and Proposition \ref{30ag2013} applied to $\bar{\cX}=\cX/N$, we have $|\Phi(S/N)|=\frac{1}{p^2}|S|/|N|$. Since $\Phi(S)/(\Phi(S)\cap N) \cong \Phi(S)N/N$, this yields $|N|\le |\Phi(S)\cap N|$. Therefore, if $|N|\leq |\Phi(S)|$ then $N$ is contained in $\Phi(S)$.
\end{proof}
\begin{proposition}
\label{pro27luglio2014} For $i=1,2$, the quotient curve $\bar{\cX}=\cX/M_i$ is rational.
\end{proposition}
\begin{proof} Every point in $\Omega_i$ is fixed by an element of $M_i$ order $p$. From the Hurwitz genus formula applied to $M_i$,
$$\textstyle\frac{p-2}{p}|S|\geq\textstyle\frac{|S|}{p}(\bar{\gg}-1)+\textstyle\frac{|S|}{p}(p-1)$$
where $\bar{\gg}$ is the genus of the quotient curve $\bar{\cX}=\cX/M_i$. This yields $\bar{\gg}=0$.
\end{proof}
\begin{proposition}
\label{pro18feb2014} For $3\le i \le p+1$, the quotient curve $\bar{\cX}=\cX/M_i$ is a curve given in (ii) of Theorem \ref{princ}, and the extension $\K(\cX)|\K(\bar{\cX})$ is an unramified $p$-extension with Galois group isomorphic to $M_i$.
\end{proposition}
\begin{proof} Since $M_i$ is semiregular on $\cX$, the extension $\K(\cX)|\K(\bar{\cX})$ is unramified. Furthermore, since $M_i$ is a subgroup of $S$ of index $p$, (\ref{eq18feb2014}) together with the Hurwitz and the Deuring-Shafarevich formulas give $\bar{\gg}-1=\bar{\gamma}-1=p-2$ where $\bar{\gg}$ is the genus and $\bar{\gamma}$ is the $p$-rank of $\bar{\cX}$.
\end{proof}
\begin{rem}
\label{14mag2014} {\em{ From Propositions \ref{pro18feb2014} and \ref{shafart}(i), every minimal generator set of $M_i$ with $3\le i \le p+1$ has size at least $2$ and at most $p-1$. We will show curves attaining this bound $p-1$.}}
\end{rem}

Theorem \ref{princ} follows from Lemmas \ref{lem17Afeb2014} and \ref{lem17feb2014}} together with Propositions \ref{propap=3}, \ref{pro12dic} and \ref{pro18feb2014}.

For the rest of the paper, $\cX$ always denotes an extremal Nakajima curve. Also, we keep our notation and terminology adopted in Section \ref{princip}. In particular,
$\gg=\gg(\cX)=(p-2)p^{n-1}+1$ and $S$ is a Sylow subgroup of $\aut(\cX)$ of order $p^n$ with its subgroups $M_1,M_2,\ldots,M_{p+1}$ of index $p$.

\section{Infinite Family of Examples}
\label{inffam}
Let $\bar{\cX}$ be a general curve of genus $p-1$ defined in Remark \ref{rem1ago2014} with function field $F=\mathbb{K}(\bar{\cX})=\mathbb{K}(x,y)$ where
\begin{equation}
\label{eq2dic10} x(y^p-y)-ax^2-1=0,\,\,\,a\in \mathbb{K}^*.
\end{equation}
For a positive integer $N$, let $F_N$ be the largest unramified abelian
extension of $F$ of exponent $N$; that is, $F_N|F$ has the following three properties:
\begin{itemize}
\item[(i)] $F_N|F$ is an unramified Galois extension;
\item[(ii)] $F_N$ is generated by all function fields which are cyclic unramified extensions of $F$ of degree $p^N$,
\item[(iii)] ${\rm{Gal}}(F_N|F)$ is abelian and  $u^{p^N}=1$ for every element $u\in {\rm{Gal}}(F_N|F)$.
\end{itemize}
{}From classical results  due to Schmid and Witt \cite{schmidwitt1938}, we have that $\deg(F_N|F)=p^{(p-1)N}$ and that ${\rm{Gal}}(F_N|F)$ is the direct product of $p-1$ copies of the cyclic group of order $p^N$.
Let $\cX$ be the curve such that $F_N=\mathbb{K}(\cX)$. Since $F_N$ is an unramified extension of $F$, the Deuring-Shafarevich formula  yields  $\gamma(\cX)-1=p^{(p-1)N}(p-2)$. Our aim is to prove that $\aut(\cX)$ contains a $p$-group of order $p^{(p-1)N+1}$.

Let $\mathbb{K}(x)$ be the rational
 subfield of $F$ generated by $x$. Obviously, $\mathbb{K}(x)$ is a subfield of $F_N$ and we are going to consider the Galois closure $M$ of $F_N|\mathbb{K}(x)$. Let  $M=\mathbb{K}(\cY)$ where $\cY$ is an algebraic curve defined over $\mathbb{K}$. Take any $\mu\in {\rm{Gal}}(M|\mathbb{K}(x))$. Then $\mu$ is a $\mathbb{K}$-automorphism of $\cY$ fixing $x$. Let $v=\mu(y)$. Since $\mu(x(y^p-y)-ax^2-1)=x(v^p-v)-ax^2-1$, from (\ref{eq2dic10})
$$x(v^p-v)-ax^2-1=0.$$  This together with (\ref{eq2dic10}) yield that either $v=y$ or $v=y+s$ with
$s\in \mathbb{F}_p^*$. In both cases $v\in F$. Therefore, ${\rm{Gal}}(M|\mathbb{K}(x))$ viewed as a subgroup $G$ of $\aut(\cY)$ preserves
$F$. From the definition of $F_N$, this implies that $G$ also preserves $F_N$. If $L$ is the (normal) subgroup of $G$ fixing $F_N$ elementwise, this yields that $H=G/L$ is a subgroup of $\aut(\cX)$.
Let $T$ be the subfield of $M$ consisting of all elements which are fixed by $L$. Since $F_N\subseteq T \subseteq M$ and $M|T$ is a Galois extension,  we have that
$$|G|=[M:\mathbb{K}(x)]=[M:T][T:F_N][F_N:F][F:\mathbb{K}(x)]=|L|[T:F_N]p^{(p-1)N}p,$$
whence $|H|=|G|/|L|$ is divisible by $p^{(p-1)N+1}$. Let $S$ be a Sylow $p$-subgroup of $H$. Then $S$ is a subgroup of $\aut(\cX)$ so that $\gamma(\cX)-1=(p-2)\textstyle\frac{|S|}{p}$. Therefore, the following result is obtained.
\begin{theorem}
\label{31mag2014} For $N\geq 1$, let $\cX$ be the curve whose function field $\K(\cX)$ is generated by all cyclic unramified $p$-extensions of degree $p^N$ of the function field of
the curve $\bar{\cX}$ with affine equation (\ref{eq2dic10}). Then $\cX$ is an extremal Nakajima curve of genus $\gg(\cX)=p^{(p-1)N}(p-2)+1$ whose $p$-group of automorphisms $S$ is a semidirect product
$U\rtimes \langle s \rangle $ where $U$ is the direct product of $p-1$ cyclic group of order $p^N$ and $s$ has order $p$.
\end{theorem}
Theorem \ref{31mag2014} together with Proposition \ref{pro12dic} provides a curve of type (iii) in Theorem \ref{princ}, for every proper power of $p$. An explicit example, for $p=3$ and $N=1$, is
given in Section \ref{S27}.

In our construction, $F_N$ may be replaced by any unramified Galois extension $F'$ such that $G={\rm{Gal}}(F'|F)$ is a finite group of order $p^m$ with $d(G)=p-1$,
whose automorphism group $\aut(G)$ attains (\ref{eq25agostoA}).
In fact, Proposition \ref{shafart} shows that $F'$ is the unique unramified Galois extension of $F$ with Galois group $G$ in the separable algebraic closure of $F$. Therefore, if $\cX$ is a curve with function field $F'$, the above argument shows that $\cX$ is a Nakajima extremal curve with $p$-rank equal to $p^{m+1}(p-2)$. This proves the following result.
\begin{theorem}
\label{14mag2014a} Let $G$ be a finite $p$-group of order $p^n$ such that the minimum size of its generator sets equals $p-1$. Assume that the automorphism group of $G$ attains (\ref{eq25agostoA}). Then, for every $a\in \K^{*}$, there exists a unique Nakajima extremal curve $\cX$ which is an unramified $p$-extension of the curve $\bar{\cX}$, as in Remark
{\rm\ref{rem1ago2014}}, with ${\rm{Gal}}(\K(\cX)|\K(\bar{\cX}))\cong G$.
\end{theorem}
{}From Remark \ref{rem15may2014}, Theorem \ref{14mag2014a} applies to the above considered direct product of $p-1$ copies of the cyclic group of order $p^N$, and to the group $UT(r,p)$ for $r=p$.
A further refinement of the above construction is given in the following theorem.
\begin{theorem}
\label{16mag2014a} Existence (but not necessarily uniqueness) of a Nakajima extremal curve stated in Theorem \ref{14mag2014a} holds true under the weaker hypothesis that a Sylow $p$-subgroup of the automorphism group of $G$ attains (\ref{eq16may2014}).
\end{theorem}
\begin{proof} Let $|G|=p^m$. In a separable algebraic closure of $F$,  
let $\{F_1,\ldots,F_k\}$ be the set of all unramified Galois extension $F_i|F$ with $G\cong{\rm{Gal}}(F_i|F)$, and let $F'$ be their compositium. Obviously, the Galois closure $M$ of $F'|\mathbb{K}(x)$ contains each $F_i$. Since $d(G)=p-1$, Corollary \ref{25agosto2014} yields that $k$ is not divisible by $p$. Our arguments leading to Theorem
\ref{14mag2014a} show that ${\rm{Gal}}(M|\K(x))$ preserves $F$, and hence leaves the set $\{F_1,\ldots,F_k\}$ invariant. Since $p\nmid k$, any $p$-subgroup of ${\rm{Gal}}(M|\K(x))$ preserves at least one of them, say $F_1$. As
$$|{\rm{Gal}}(M|\K(x))|=[M:F'][F':F_1][F_1:F][F:\mathbb{K}(x)]=[M:F'][F':F_1]p^{m+1},$$
${\rm{Gal}}(M|\K(x))$ has a subgroup of order $p^{m+1}$ that preserves $F_1$. This shows that if $\cX$ is a curve with $\K(\cX)=F_1$, then $\aut(\cX)$ has a subgroup of order $p^{m+1}$. Since $[F_1:F]$ is an unramified Galois extension with Galois group of order $p^m$ and $\bar{\cX}$ has $p$-rank $p-1$, the Deuring-Shafarevich formula yields that $\cX$ has $p$-rank $p^m(p-2)+1$. Therefore, $\cX$ is a Nakajima extremal curve with an automorphism group of order $p^{m+1}$. Our argument also shows that uniqueness might not hold when $k\not\equiv 1 \pmod p$.
\end{proof}
With some changes, the above construction also applies to the Artin-Mumford curve $\cM_c=\bar{\cX}$ with affine equation (\ref{artinmumford}). As we have already mentioned,
$\gg(\bar{\cX})=\gamma(\bar{\cX})=(p-1)^2$ and $\aut(\bar{\cX})$ has an elementary abelian subgroup of order $p^2$ generated by $\alpha=(x,y)\rightarrow (x+1,y)$ and $\beta=(x,y)\rightarrow (x,y+1)$. In fact, if $F=\mathbb{K}(t)$ is the rational field generated by $t=x^p-x$, and $M$ is the Galois closure of $F_N|\mathbb{K}(t)$ then every $\mu\in {\rm{Gal}}(M|\mathbb{K}(t))$ preserves
the Artin-Mumford curve $\bar{\cX}$. Therefore, the following result holds.
\begin{theorem}
\label{20luglio2014} For $N\geq 1$, let $\cX$ be the curve whose function field $\K(\bar{\cX})$ is generated by all cyclic unramified $p$-extensions of degree $p^N$ of the function field of
the Artin-Mumford curve $\bar{\cX}$ with affine equation (\ref{artinmumford}). Then $\cX$ is an extremal Nakajima curve of genus $\gg(\cX)=p^{N(p-1)^2+1}(p-2)+1$ with a $p$-group of automorphisms $S$ whose Frattini subgroup $\Phi(S)$ of order $p^{N(p-1)^2}$ is the direct product of $(p-1)^2$ copies of the cyclic group of order $p^N$, so that the factor group $S/\Phi(S)$ is elementary abelian of order $p^2$.
\end{theorem}
\section{The structure of $S$ for $|S|\leq p^{p+1}$}
\begin{proposition}
\label{a31ago2013}
If $|S|=p^3$ then $S$ isomorphic to $UT(3,p)$, the unique non-abelian group of order $p^3$. Furthermore, the non-trivial elements of $S$ which have fixed points are at most $2(p^2-p)$.
\end{proposition}
\begin{proof} From the classification of groups of order $p^3$, see \cite[Chapter I, 14.10 Satz]{huppertI1967}, either $S=C_{p^2}\rtimes C_p$, or $S\cong UT(3,p)$. Since the group $C_{p^2}\rtimes C_p$ has more than two cyclic maximal subgroups, the first assertion follows from Proposition \ref{31ago2013}. The elements of $S$ with fixed points fall into two subgroups, namely $M_1$ and $M_2$, both elementary abelian of order $p^2$. Since $Z(S)$ is a subgroup of $M_1$ of order $p$, Proposition \ref{propgp=3} shows that $M_1$ (and $M_2$) has at most as many as $p^2-p$ non-trivial elements with a fixed points.
\end{proof}
\begin{proposition}
\label{14may2014d}
For $c\in \K^*$, the curve  $\cX_c$ with function field $\mathbb{K}(x,y,z)$ defined by the equations
\begin{itemize}
\item[\rm(i)] $(x^p-x)(y^p-y)-c=0$;
\item[\rm(ii)] $z^p-z+x^py-xy^p=0$.
\end{itemize}
is a Nakajima extremal curve whose automorphism group has order $p^3$,  and its $\mathbb{K}$-automorphism group is a semidirect product of $U(p,3)$ by a dihedral group of order $2(p-1)$.
\end{proposition}
\begin{proof}
As before, let $\cM_c$ denote the Artin-Mumford curve with affine equation (\ref{artinmumford}). We first show that $\mathbb{K}(\cX_c)$ is an unramified Artin-Schreier extension of $\mathbb{K}(\cM_c)$.
This will imply that $\gg(\cX_c)=\gamma(\cX_c)=(p-2)p^2+1$.

Since $\gg(\cM_c)=(p-1)^2$ and $\mathbb{K}(\cM_c)=\mathbb{K}(x,y)$ with $x,y$ as in (\ref{artinmumford}), there exist places $P_0,\ldots,P_{p-1},$ $Q_0,\ldots,Q_{q-1}$ such that
$$
(y)_0=pP_0,\qquad (y)_{\infty}=Q_0+\ldots+Q_{q-1},
\qquad (x)_0=pQ_0,\qquad (x)_{\infty}=P_0+\ldots+P_{p-1},
$$
and for each $i=1,\ldots,p-1$
$$
v_{P_i}(y-i)=v_{Q_i}(x-i)=p.
$$
Let $u=xy^p-x^py$. Then $u=xy\prod_{a \in \mathbb F_P^\star}(y-ax)$. The pole divisor of $u$ is
$$
(u)_\infty=p(P_1+\ldots+P_{p-1}+Q_1+\ldots+Q_{p-1}).
$$
Also,
$$
v_{P_0}(u)=0,\qquad v_{Q_0}(u)=0.
$$
In order to prove that the equation $z^3-z=u$ defines an Artin-Schreier extension of $\mathbb K(x,y)$, we first show that $u\neq w^p-w$ for every $w\in \mathbb K(x,y)$; see \cite[Proposition III.7.8]{STI2ed}.
A canonical divisor of $\mathbb K(x,y)$ is
$$
W=(p-2)(P_0+\ldots+P_{p-1}+Q_0+\ldots+Q_{p-1}),
$$
and a $\mathbb K$-basis of ${\mathcal{L}}(W)$ is
$$
\{x^iy^j\mid 0\le i\le p-2,\, 0\le j\le p-2\}.
$$
Assume that $u=w^p-w$ for some $w\in \mathbb K(x,y)$. Then
$$
(w)_\infty=P_1+\ldots+P_{p-1}+Q_1+\ldots+Q_{p-1}.
$$
Therefore, $w\in {\mathcal{L}}(W)$, and hence
$$
w=\sum_{i=0,\ldots,p-1}x^if_i(y),
$$
for $f_i$ a polynomial in $\mathbb K[T]$ of degree less than or equal to $p-2$. Note that for each $k=1,\ldots,p-1$
$$
v_{P_k}(x^if_i(y))=-i+ps_{i,k},
$$
where $s_k$ is the multiplicity of $k$ as a root of $f_i$. As the degree of $f_i$ is less than $p-1$, for each $i>0$ with $f_i(y)\neq 0$ there is some $k$ with $s_{i,k}=0$. Let $k_i$ be the minimum of such $k$'s. Then
$$
-1=v_{P_{k_i}}(w)=-i,
$$
which shows that $f_i(y)=0$ for each $i\ge 2$.
Then
$$
w=f_0(y)+xf_1(y).
$$
Analogously, it can be proved that
$$
w=g_0(x)+yg_1(x)
$$
for some polynomials $g_0,g_1\in \mathbb K[T]$ of degree less than or equal to $p-2$.
The only possibility is that
$$
w=\alpha+\beta x+\gamma y+\delta xy, \text{ for some }\alpha,\beta,\gamma,\delta \in \mathbb K.
$$
Therefore,
$$
u=xy^p-x^py=(\alpha+\beta x+\gamma y+\delta xy)^p-(\alpha+\beta x+\gamma y+\delta xy)=\alpha^p-\alpha-\beta x+\beta^px^p-\gamma y+\gamma^py^p-\delta xy+\delta^px^py^p.
$$
If $\beta\neq 0$, then
$$
v_{P_0}(u)=v_{P_0}(\beta^px^p)=-p;
$$
similarly, if $\gamma\neq 0$ then
$$
v_{Q_0}(u)=v_{Q_0}(\gamma^py^p)=-p.
$$
As $v_{P_0}(u)=v_{Q_0}(u)=0$, we have
 $\beta=\gamma=0$ and hence $u=\alpha^p-\alpha-\delta xy+\delta ^px^py^p$.
 From $(x^p-x)(y^p-y)=c$ it follows
$
x^py^p=x^py+xy^p-xy+c,
$
whence $u=\delta^p(x^py+xy^p-xy+c)-\delta xy+\alpha^3-\alpha$, and
$$
(1-\delta^p)xy^p-(1+\delta^p)x^py+(\delta^p+\delta)xy-(\delta^p c+\alpha^p-\alpha)=0.
$$
Valuating at $P_1$ and $Q_1$ gives $\delta^p=1$ and $\delta^p=-1$, a contradiction.

In order to prove that that the extension $\mathbb{K}(x,y,z)|\mathbb{K}(x,y)$ is unramified,
we need to show that
for each $i=1,\ldots,p-1$ there exist  $t_i$ and $v_i$  such that
\begin{equation}\label{key}
v_{P_i}(xy^p-x^py-(t_i^p-t_i))\ge 0,\qquad v_{Q_i}(xy^p-x^py-(v_i^p-v_i))\ge 0.
\end{equation}
Let $t_i=ix$.
Then
$$
xy^p-x^py-(t_i^p-t_i)=xy^p-x^py+i x^p-i x=x^p(i-y)-x(i-y)^p=x(i-y)\prod_{a \in \mathbb F_p^\star}(x-a(i-y))
$$
and hence
$$
v_{P_i}(xy^p-x^py-(t_i^p-t_i))=v_{P_i}(y-i)-p=0.
$$
Similarly, one can show that
$v_{Q_i}(xy^p-x^py-((iy)^p-(iy)))=0$ for each $i=1,\ldots,p-1$. This completes the proof of the first assertion.

Both maps
$$ g:\,(x,y,z)\mapsto (x+1,y,z+y)\quad h:\,(x,y,z)\mapsto (x,y-1,z+x)$$
are in $\aut(\cX)$. They generate a non-abelian group $S$ of order $p^3$ and exponent $p$. Therefore $S\cong UT(p,3)$. Furthermore, $\aut(\cX)$ contains the maps $r:\,(x,y,z)\mapsto (y,x,-z)$, and $t:=(x,y,z)\mapsto(\omega x,\omega^{-1} y,z)$ where $\omega$ is primitive element of $\mathbb{F}_p$.  By a straightforward computation, $\langle r,t \rangle\cong D_{p-1}$ and
$$rgr=h^{-1},\, rhr=g^{-1},\, t^{-1}gt=g^{\omega^{-1}},\, t^{-1}ht=h^{\omega}.$$ Thus $G=\langle g,h,r,t \rangle\cong U(p,3)\rtimes D_{p-1}$. Actually $G$ is the full $\mathbb{K}$-automorphism group of $\cX$ for $p>3$. This follows from Theorem \ref{fullaut}. For $p=3$, a Magma computation shows that $\aut(\cX)$ is larger as it has order $432$ and $\aut(\cX)\cong U(3,3)\rtimes V$ where $V$ is a semidihedral group of order $16$.
\end{proof}
\begin{proposition}
\label{18mar2014}
If $|S|\leq p^p$ then $S$ has exponent $p$.
\end{proposition}
\begin{proof} From \cite[Chapter III, 10.2 b) Satz]{huppertI1967}, $S$ is a regular $p$-group. By (v) of Proposition \ref{30ag2013}, $S$ is generated by (two) elements of order $p$. Therefore,
the subgroup $\Omega_1(S)$ generated by all elements of order $p$ is the whole group $S$. From  \cite[Chapter III, 10.7 a) Satz]{huppertI1967}, the subgroup of $S$ generated by all elements which are proper $p$-powers of elements in $S$ is trivial. Hence, every non-trivial element of $S$ has order $p$.
\end{proof}
\begin{proposition}
\label{aprop15may2014} If $|S|=p^{p+1}$, then $S$ has exponent $p$ or $p^2$. In the latter case, $M_1$ and $M_2$ have exponent $p$, and if $M_i$ with $3\le i \le p+1$ has exponent $p^2$ then all elements of $M_i$ of order $p$ are in $\Phi(S)$. Moreover, the maximal normal subgroups $M_i$ of exponent $p^2$ are as many as $k$, then the number of elements of $S$ of order $p$ is equal to $(p+1-k)(p^p-p^{p-1})+p^{p-1}-1$.
\end{proposition}
\begin{proof} The subgroup $N_1$ generated by the elements of $M_1$ of order $p$ is a characteristic subgroup of $M_1$. Since $M_1$ is a normal subgroup of $S$, this yields that $N_1$ is a normal subgroup of $S$. By Lemma \ref{eq130apr2013}, the stabilizer of a point $P\in \Omega_1$ is in $N_1$. Hence Proposition \ref{propcp=3} yields $N_1=M_1$.  Since $M_1$ has order $p^p$ its exponent is equal to $p$. Therefore, \cite[Chapter III, 10.7 a) Satz]{huppertI1967} yields no non-trivial element of $M_1$ is a $p$-power of an element of $M_1$, that is, $M_1$ has exponent $p$. This remains true for $M_2$. If $S$ has exponent $p^h$ with $h>1$ then some $M_i$ with $3\le i \le p+1$ contains an element $u$ of order $p^i$. Since $\Phi(S)$ is a subgroup of $M_i$ of index $p$, $\Phi(S)$ contains $u^p$. On the other hand $\Phi(S)$ is a subgroup of $M_1$ and $M_1$ has exponent $p$. Therefore, $u^{p^2}=1$ whence $h=2$. Moreover, if $M_i$ had an element $v$ of order $p$ other than those in $\Phi(S)$, then $\Phi(S)$ together with $v$ would generate $M_i$. Since $M_i$ is a $p$-regular subgroup, this would yield $M_i$ to have exponent $p$, again by \cite[Chapter III, 10.7 a) Satz]{huppertI1967}; a contradiction. Therefore, no element of $M_i\setminus \Phi(S)$ has order $p$. If we have $k$ such $M_i$, then $S$ has exactly $(p+1-k)(p^p-p^{p-1})+p^{p-1}-1$ whence the last claim follows.
\end{proof}
\section{Particular families of groups}
Metacylic, regular $p$-groups and $p$-groups with maximal nilpotency class play an important role in Group theory; the main references are  \cite[Section III.14]{huppertI1967}, and \cite{berkovichbook}. This gives a motivation for the study of Nakajima extremal curves whose $p$-automorphism group $S$ falls in one of those families.
\begin{proposition}
\label{d4set2013} If $|S|\ge p^4$ then $S$ is not metacyclic.
\end{proposition}
\begin{proof} Assume on the contrary that $S$ is metacyclic. From Proposition \ref{30ag2013} and \cite[Lemma 2.2]{blackburn1958},
$S'/S$ is cyclic. Therefore $S'$ contains a characteristic subgroup $N$ of index $p$. By (i) of Proposition \ref{30ag2013}, $N$ has index $p^3$ in $S$. From  Proposition \ref{a30ag2013} applied to $N$, $\bar{S}=S/N$ is a subgroup of $\aut(\bar{\cX})$ with $\bar\cX=\cX/N$ such that $|\bar{S}|=p^3$, Proposition \ref{a31ago2013} implies that $\bar{S}\cong UT(3,p)$. On the other hand, as  $S$ is metacyclic,  \cite[Theorem 2]{berkovich2006} yields that $\bar{S}={S/N}$ is also a metacyclic group. But  $UT(3,p)$ is not a metacyclic group by Proposition \ref{a31ago2013},
a contradiction.
\end{proof}
\begin{proposition}
\label{9mag2014} $S$ is a regular $p$-group if and only if $S$ has exponent $p$.
\end{proposition}
\begin{proof} The proof of Proposition \ref{18mar2014} shows that if $S$ is regular then it has exponent $p$. The converse also holds, see \cite[Chapter III, 10.2 d) Satz]{huppertI1967}.
\end{proof}
\begin{proposition}
\label{31ago2013} If $|S|>p^2$ then none of the subgroups $M_i$ is cyclic.
\end{proposition}
\begin{proof} For $i=1,2$ the assertion follows from Proposition \ref{propcp=3}. For $3\le i \le p+1$ the proof is by induction on $|S|$. In the smallest case,  $|S|=p^3$, the assertion is a
consequence of Proposition \ref{a31ago2013}. Assume that $M=M_i$ is cyclic for some  $3\le i \le p+1$. Let $T$ be the unique subgroup of $M$ of order $p$. Since $M$ is a normal subgroup of $S$,
$T$ is a normal subgroup of $S$, as well. As $T$ is semiregular, the quotient curve $\bar{\cX}=\cX/T$ is a Nakajima extremal curve with Sylow $p$-subgroup $S/T$. Since $|S/T|=\textstyle\frac{1}{p}|S|$ and  $|M/T|=\textstyle\frac{1}{p}|M|$, the inductive hypothesis yields that $M/T$ is not cyclic. But then $M$ itself is not cyclic.
\end{proof}
\begin{proposition}
\label{proA18feb2014} If at least two of the $p+1$  maximal normal subgroups $M_i$ of $S$ are abelian then $|S|=p^2$ or $|S|=p^3$.
\end{proposition}
\begin{proof} Assume that $|S|\neq p^2$. From \cite[Chapter I, Aufgabe 21)]{huppertI1967}, every $p$-group with at least two abelian maximal normal subgroup has class at most $2$. On the other hand, if a non abelian group $G$ of order $p^n$ has an abelian maximal normal subgroup and the commutator subgroup of $G$ has index $p^2$ then $G$ has (maximal) class $n-1$; see \cite[Theorem 2.5]{xu2008}. This applies to $S$ in our case by (i) and (ii) of Proposition \ref{30ag2013}. Therefore, $n-1=2$.
\end{proof}
The result on $G$   quoted in the proof of Proposition \ref{proA18feb2014} together with (i) and (ii) of Proposition \ref{30ag2013} also give the following result.
\begin{proposition}
\label{a4set2013} If $M_i$ is abelian for some $3\le i\le p+1$, then $S$ has maximal nilpotency class.
\end{proposition}
The subgroup $U$ in Theorem \ref{31mag2014} is an abelian subgroup of $S$ of index $p$. Therefore, the proof of Proposition \ref{proA18feb2014} can be used to prove the first assertion.
\begin{proposition}
\label{31mag2014a} The $p$-automorphism group $S$ of the Nakajima extremal curve given in Theorem {\rm \ref{31mag2014}} has maximal nilpotency class.
\end{proposition}
\begin{proof} The subgroup $U$ in Theorem \ref{31mag2014} is an abelian subgroup of $S$ of index $p$. Therefore, the proof of Proposition \ref{proA18feb2014} can be used to prove the assertion.
\end{proof}
\begin{rem}{\em{According to Proposition \ref{pro12dic}, the quotient curves of the curve given in Theorem \ref{31mag2014} are also Nakajima extremal curves. Their $p$-automorphism groups have maximal nilpotency class, as well, by \cite[Section III, 14.2 Hilfssatz]{huppertI1967} }}.
\end{rem}
\begin{proposition}
\label{31mag2014b} The $p$-automorphism group $S$ of the Nakajima extremal curve given in Theorem {\rm\ref{20luglio2014}} has no maximal nilpotency class.
\end{proposition}
\begin{proof} From Theorem \ref{20luglio2014}, the minimum size of a generator set of $\Phi(S)$ is $(p-1)^2$. Since $(p-1)^2>p-1$, $\Phi(S)$ cannot be generated by $p-1$ elements. If $S$ has maximal nilpotency class, this implies that $S$ must be of order $p^{p+1}$ and isomorphic to the Sylow $p$-subgroup of the symmetric group of degree $p^2$, see \cite[Theorem 5.2]{berkovich2002}. Since $|S|=p^{N(p-1)^2+2}$, this yields $N(p-1)^2+2=p+1$, a contradiction which proves the assertion.
\end{proof}
By \cite[Chapter III, 14.22 Satz]{huppertI1967}, any $p$-group of maximal nilpotency class and order bigger than $p^{p+1}$ has exactly one maximal subgroup which is a regular $p$-group. This subgroup, called the \emph{fundamental subgroup}, plays a relevant role in the study of $p$-groups.
\begin{proposition}
\label{31mag2014c} Let $S$ be the $p$-automorphism group of a Nakajima extremal curve such that $S$ has maximal nilpotency class and order bigger than $p^{p+1}$. If $s\in S$ is an element of order $p$ then number of fixed points of $s$ is either zero, or $p$. Accordingly, the relative quotient curve $\cZ=\cX/\langle s \rangle$  of $\cX$ has genus
\begin{equation}
\label{eq31mag2014}
\gg(\cZ)= \left\{
\begin{array}{lll}
(p-2)p^{n-2}+1,\\
(p-2)p^{n-2}-(p-1)+1.
\end{array}
\right.
\end{equation}
\end{proposition}
\begin{proof} If $s$ has no fixed point in $\Omega$, then the Deuring-Shafarevich formula shows that
$\gg(\cZ)=(p-2)p^{n-2}+1$. Therefore, we focus on an element $s\in S$ which fixes a
point in $\Omega$. Then $s\in M_1$ or $s\in M_2$, according as the set $\Omega_s$ of the fixed points of $s$ is contained in $\Omega_1$ or in $\Omega_2$. Assume that $\Omega_s\subset \Omega_1$, and let $P_1,P_2$ be any two distinct points in $\Omega_s$. Since $\Omega_1$ is an $S$-orbit, there exists $h\in S$ that takes $P_1$ to $P_2$. Then $hsh^{-1}$ fixes $P_1$, and Lemma \ref{eq130apr2013} implies that either $hsh^{-1}=s$ or $hsh^{-1}=s^{-1}$. The latter case cannot actually occur as in a $p$-group a non-trivial element and its inverse are in different conjugacy classes. Therefore, $h$ is in the centralizer $C_S(s)$ of $s$. The converse also holds. Thus $p|\Omega_s|=|C_S(s)|$.

We show that the fundamental subgroup of $S$ is neither $M_1$ nor $M_2$. Assume on the contrary that it is $M_1$. The argument at the beginning of the proof of Proposition \ref{aprop15may2014} shows that $M_1$ is generated by its elements of order $p$. Since $M_1$ is a regular $p$-group, \cite[Chapter III, 10.7 a) Satz]{huppertI1967} shows that $M_1$ has exponent $p$.  Now, the last claim of \cite[Chapter III, 14.16 Satz]{huppertI1967} yields $|M_1|=p^{p-1}$, a contradiction. Therefore, one of the other subgroups, say $M_3$, is the fundamental subgroup of $S$, and $s\in S\setminus M_3$. By \cite{blackburn1958b}, see also \cite[Remark 4]{berkovichbook}, this yields that $|C_S(s)|=p^2$. Hence, $|\Omega_s|=p$. Finally, the Deuring-Shafarevich formula shows that $\gg(Z)=(p-2)p^{n-2}-(p-1)+1$.
\end{proof}
The converse of Proposition \ref{31mag2014c} also holds.
\begin{proposition}
\label{31mag2014c1} Let $S$ be the $p$-automorphism group of a Nakajima extremal curve with $|S|=p^n,\, n\geq3$. If some element $s\in S$ has exactly $p$ fixed points, then $S$ has maximal nilpotency class.
\end{proposition}
\begin{proof} The first part of the proof of Proposition \ref{31mag2014c} also shows that if an element $s\in S$ has exactly $p$ fixed points then $|C_S(s)|=p^2$. The latter condition means that
the conjugacy class of $s$ in $S$ has size $p^{n-2}$. Therefore, the claim follows from  \cite[Chapter III, 14.23 Satz]{huppertI1967}.
\end{proof}
\section{Proof of Theorem \ref{fullaut}}
\begin{lemma}
\label{lemagosto9C} Let $N$ be a normal subgroup of $\aut(\cX)$ such that the quotient curve $\bar{\cX}=\cX/N$ is neither rational nor elliptic. Then the order of $N$ is a power of $p$. Furthermore,  $\bar{\cX}$ is an extremal Nakajima curve provided that its genus is bigger than $p-1$.
\end{lemma}
\begin{proof} Let $|N|=ap^b$ with $a$ prime to $p$. We may assume that $S\cap N$ is a Sylow subgroup of $N$. From the Hurwitz genus formula applied to $N$, $\gg-1=p^{n-1}(p-2)\geq a p^b(\bar{\gg}-1)$. On the other hand, since $SN/N\cong S/S\cap N$ is a $\mathbb{K}$-automorphism group of the quotient curve $\bar{\cX}=\cX/N$ whose order is $p^{n-b}$, the Nakajima bound gives $p^{n-b-1}(p-2)\leq \bar{\gg}-1.$
Then,
$$\textstyle\frac{p-2}{a}p^{n-1-b}\geq \bar{\gg}-1\geq p^{n-1-b}(p-2).$$
Therefore $a=1$ and this proves the assertion.
\end{proof}
\begin{lemma}
\label{lemagosto9D} Let $N$ be a normal subgroup of $\aut(\cX)$ such that the quotient curve $\bar{\cX}=\cX/N$ is rational. Then the order of $N$ is a divisible by $p^{n-1}.$
\end{lemma}
\begin{proof} By Proposition \ref{propap=3} $S$ has two short orbits, $\Omega_1$ and $\Omega_2$, both of size $p^{n-1}$. Since $S$ normalizes $N$, the Hurwitz genus formula applied to $N$ gives
$$2\gg-2=2(p-2)p^{n-1}=-2|N|+p^{n-1}(d_P+d_Q)+\kappa p^n$$
with $P\in \Omega_1,\,Q\in\Omega_2$ and $\kappa$ a non-negative integer. From this the assertion follows.
\end{proof}
To obtain a similar result for the case where $\bar{\cX}$ is elliptic, we need some technical results.
\begin{lemma}
\label{lem7agosto2014}
Assume that $S$ is not a normal subgroup of $\aut(\cX)$ and that $T$ is a Sylow $p$-subgroup of $\aut(\cX)$ other than $S$. If there exists a point $P\in \Omega_1$ fixed by a non-trivial element of $T$ then no point in $\Omega_2$ is fixed by a non-trivial element of $T$.
\end{lemma}
\begin{proof} Let $G=\aut(\cX)$. In $G_P$, all $\mathbb{K}$-automorphisms of order a power of $p$ lie in the first ramification group $G_P^{(1)}$. Obviously, $G_P^{(1)}$ contains both $S_P$ and $T_P$. Actually $S_P=T_P$ must hold by virtue of Lemma \ref{eq130apr2013} applied to a Sylow $p$-subgroup of $\aut(\cX)$ containing $G_P^{(1)}$.
Assume on the contrary the existence of a point $Q\in \Omega_2$ fixed by a non-trivial element of $T$. As before this yields $S_Q=T_Q$. Hence $\langle S_P,S_Q \rangle =\langle T_P,T_Q \rangle$. By (v) of Proposition \ref{30ag2013}, $S=\langle S_P,S_Q \rangle$. Therefore, $S\leq T$. Since $S$ and $T$ are  Sylow $p$-subgroups of $\aut(\cX)$, this yields $S=T$.
\end{proof}
\begin{lemma}
\label{lem8agostoA2014}
If a Sylow $p$-subgroup $T$ of $\aut(\cX)$ preserves $\Omega_1\cup \Omega_2$ then it does both $\Omega_1$ and $\Omega_2$.
\end{lemma}
\begin{proof} We may assume that $T\neq S$. The assertion follows from Lemma \ref{lem7agosto2014}.
\end{proof}
\begin{lemma}
\label{lem8agosto2014}
Assume that $S$ is not a normal subgroup of $\aut(\cX)$.
If $\Omega_1$ is preserved by all Sylow $p$-subgroups of $\aut(\cX)$ then $M_1$ is a normal subgroup of $\aut(\cX).$
\end{lemma}
\begin{proof}
Let $T$ be any Sylow $p$-subgroup of $\aut(\cX)$ other than $S$. From the proof of Lemma \ref{lem7agosto2014}, $S_P=T_P$ for every point $P\in \Omega_1$. Since $M_1$ is generated by all stabilizers $S_P$ with $P$ ranging over $\Omega_1$, this shows that $M_1$ is a subgroup of $T$. Therefore, all the Sylow $p$-subgroups share $M_1$. Since $M_1$ has index $p$ in
$S$, $M_1$ is their complete intersection. From this the assertion follows.
\end{proof}

\begin{lemma}
\label{lem8agostoB2014}
Let $N$ be a normal subgroup of $\aut(\cX)$. Let $\Pi$ be the set of all points of $\cX$ which are fixed by some non-trivial element of $N$.
Assume that $S$ is not a normal subgroup of $\aut(\cX)$.
If
$0<|\Pi|< p^{n}$ then $\Pi=\Omega_1$ (or $\Pi=\Omega_2$) and $M_1$ (or $M_2$) is a normal subgroup of $\aut(\cX)$.
\end{lemma}
\begin{proof} Since $N$ is normal, $\Pi$ is partitioned in orbits of $\aut(\cX)$. In particular, the orbit of $P\in \Pi$ under the action of any Sylow $p$-subgroup of $\aut(\cX)$ is contained in $\Pi$. If $|\Pi|\le p^{n-1}$ then  $\Pi=\Omega_1$ (or $\Pi=\Omega_2$), and all Sylow $p$-subgroup of $\aut(\cX)$ preserve $\Omega_1$ (or $\Omega_2$). Therefore, the assertion follows from Lemma \ref{lem8agosto2014}. If $p^{n-1}<|\Pi|< p^{n}$, then  $\Pi=\Omega_1\cup \Omega_2$, and both $M_1$ and $M_2$ are normal subgroups of $\aut(\cX)$ by Lemmas \ref{lem8agostoA2014} and \ref{lem8agosto2014}.
But then $S=\langle M_1,M_2\rangle$ would be normal in $\aut(\cX)$, a contradiction.
 \end{proof}

\begin{lemma}
\label{lemagosto9E} Let $N$ be a normal subgroup of $\aut(\cX)$ such that the quotient curve $\bar{\cX}=\cX/N$ is elliptic.
Assume that $S$ is not a normal subgroup of $\aut(\cX)$.
If the order of $N$ is prime to $p$ then $M_1$ (or $M_2$) is a normal subgroup of $\aut(\cX)$.
\end{lemma}
\begin{proof} Since $|N|$ is prime to $p$, $S$ can be regarded as a $\mathbb{K}$-automorphism group of $\bar{\cX}$. For $P\in \Omega_1\cup \Omega_2$, let $\bar{P}$ be the point of the quotient curve
$\bar{\cX}=\cX/N$ lying under $P$. Since $S_P$ has order $p$ by Lemma \ref{eq130apr2013}, the point $\bar{P}$  is fixed by a $\mathbb{K}$-automorphism of order $p$. As $p$ is odd and $\bar{\cX}$ is elliptic, we have $p=3$; see \cite[Theorem 11.84]{hirschfeld-korchmaros-torres2008}. From the Hurwitz genus formula applied to $N$,
$$\gg-1=3^{n-1}=3^{n-1}\textstyle\frac{1}{2}(d_P+d_Q)+\frac{\tau}{2}\, 3^n$$
with $P\in \Omega_1,\,Q\in\Omega_2$ and $\tau$ a non-negative integer. This is only possible when $\tau=0$ and $d_P+d_Q=2$. Therefore, either $\Omega_1$, or $\Omega_2$, or
$\Omega_1\cup \Omega_2$ coincide with the set of all points of $\cX$ which are fixed by some non-trivial element of $N$. Now, the assertion follows from Lemma \ref{lem8agostoB2014}.
\end{proof}
\begin{lemma}
\label{lemagosto10} For an odd prime $d$ other than $p$, let $U$ be a $d$-subgroup of $\aut(\cX)$ of order $d^u$ and exponent $d^e$. Then  $d^{u-e}$ divides $p-2$.
\end{lemma}
\begin{proof} 
If $U$ has no short orbit, then $d^u$ divides $\gg-1$ by the Hurwitz genus formula applied to $U$, and the assertion follows. We may assume that $U$ has $m\ge 1$ short orbits and let $\ell_1,\ldots,\ell_m$ be their lengths. From  the Hurwitz genus formula applied to $U$,
\begin{equation}
\label{eqagosto10}
2\gg-2=2(p-2)p^{n-1}=d^u(2\bar{\gg}-2)+\sum_{i=1}^m(d^u-\ell_i)
\end{equation}
where $\bar{\gg}$ is the genus of the quotient curve $\bar{\cX}=\cX/U$.
Let $P$ be a point from a short orbit of length $\ell_i$. Then $d^u=|U_P|\ell_i$. Since $U_P$ is a cyclic subgroup of $U$, we also have that $|U_P|=p^{u_i}\le p^e$. Therefore, $\ell_i=d^{u-u_i}$ with $u_i\leq e$. From (\ref{eqagosto10}), $$2(p-2)p^{n-1}=d^{u-e}(d^e(2\bar{\gg}-2)+\sum_{i=1}^m(d^{e}-d^{e-u_i}))$$
whence the assertion follows.
\end{proof}
\begin{lemma}
\label{le12agostoA} For $|S|=p^2$, one of the following cases occurs.
\begin{itemize}
\item[\rm(i)] $\cX$ is an Artin-Mumford curve with affine equation {\rm(\ref{artinmumford})}, and $\aut(\cX)$ is the semidirect product of $S$ by a dihedral group of order $2(p-1)$.
\item[\rm(ii)] $M_1$ (and $M_2)$ is a normal subgroup of $\aut(\cX)$, and $\aut(\cX)$ is the semidirect product of $S$ by a  subgroup of a cyclic group of order $p-1$.
\end{itemize}
\end{lemma}
\begin{proof}
Let $\bar{\cX}=\cX/M_1$. By Proposition \ref{pro27luglio2014}, $\mathbb{K}(\cX)|\mathbb{K}(\bar{\cX})$ is an Artin-Schreier extension. Therefore,
 since $|M_1|=p$, $M_1$ is a normal subgroup $\aut(\cX)$ with four exceptions by a result of Madan and Valentini \cite{maddenevalentini1982}; see also \cite[Theorem 11.93]{hirschfeld-korchmaros-torres2008}. One exception is given in case (i). Two of the other three exceptions have zero $p$-rank, while the forth has genus $2$, and hence they cannot actually occur in our case.

The above argument holds true for $M_2$, and hence we may assume that both $M_1$ and $M_2$ are normal subgroups of $\aut(\cX)$. Since $S$ is generated by $M_1$ and $M_2$, it turns out that $S$ is also a normal subgroup of $\aut(\cX)$. By Proposition \ref{pro27luglio2014}, the quotient curve $\bar{\cX}=\cX/M_1$ is rational. Therefore
$\aut(\cX)/M_1$ is isomorphic to a subgroup $\Lambda$  of $PGL(2,\mathbb{K})$.  Furthermore,
$S/M_1$ is isomorphic to a normal subgroup of $\Lambda$ of order $p$. Also, $p^2\nmid |\Lambda|$,
since $S$ is a Sylow $p$-subgroup of $\aut(\cX)$.
From  the classification of subgroups of $PGL(2,\mathbb{K})$, see \cite[Chapter II. Hauptsatz 8.27]{huppertI1967} and \cite{maddenevalentini1982}, $|\Lambda|=pm$ with $m|(p-1)$ and hence
 $\Lambda$ is a semidirect product of $S/M_1$ by a cyclic group $L$ of order
 $m$. Therefore, $\aut(\cX)/S$ is
isomorphic to $L$ and the assertion is proven.
\end{proof}
\begin{rem} \em{The property of $\aut(\cX)$ given in (i) of Lemma \ref{le12agostoA} characterizes the Artin-Mumford curve; see \cite{AK}.}
\end{rem}

\begin{lemma}
\label{lem6agos} Any $2$-subgroup of $\aut(\cX)$  has a cyclic subgroup of index $2$.
\end{lemma}
\begin{proof} Let $U$ be a subgroup of $\aut(\cX)$ of order $d=2^u\geq 2$. From  the Hurwitz genus formula applied to $U$,
$$2\gg-2=2(p-2)p^{n-1}=2^u(2\bar{\gg}-2)+\sum_{i=1}^m(2^u-\ell_i)$$
where $\bar{\gg}$ is the genus of the quotient curve $\bar{\cX}=\cX/U$ and $\ell_1,\ldots,\ell_m$ are the short orbits of $U$ on $\cX$. Since $2(p-2)p^{n-1}\equiv 2 \pmod 4$ while
$2^u(2\bar{\gg}-2)\equiv 0 \pmod 4$, some $\ell_i$ ($1\le i \le m$) must be either $1$ or $2$.  Therefore, $U$ or a subgroup of $U$ of index $2$ fixes a point of $\cX$ and hence
is cyclic.
\end{proof}
\begin{rem}
\label{2sub} {\em{From Lemma \ref{lem6agos} and \cite[Chapter I, Satz 14.9]{huppertI1967}, any $2$-subgroup of $\aut(\cX)$ is
either cyclic, or abelian with a cyclic subgroup of index $2$, or generalized quaternion, or dihedral, or semidihedral, or type (3) with Huppert's notation \cite{huppertI1967}. This together with deep results from Group theory, see \cite{abg,gw1,gw2,walter1969} yields that if $G$ is a non-abelian simple subgroup of $\aut(\cX)$, then
a Sylow $2$-subgroup of $G$ is either dihedral, or semidihedral. In the former case, $G\cong PSL(2,q)$, with $q\geq 5$ or $G\cong {\rm{Alt}_7}$ (the Gorenstein-Walter theorem); in the latter case,
$G\cong PSL(3,q)$ with $q\equiv 3 \pmod 4$, or $G\cong PSU(3,q)$ with $q \equiv 1 \pmod 4$, or $G=M_{11}$, where $q$ is an odd prime power (the Alperin-Brauer-Gorenstein theorem).}}
\end{rem}
We are going to investigate the possibilities of the existence of a simple normal subgroup $N$ in $\aut(\cX)$, as  described in Remark \ref{2sub}. For our purpose, it will be sufficient to consider the cases
when the quotient curve $\cX/N$ is rational. Under this hypothesis, $p$ divides $|N|$. In fact, otherwise $S$ is an abelian $p$-subgroup of $PGL(2,\mathbb{K})$, and hence $n=2$ by Proposition \ref{propbp=3}, while $\aut(\cX)$ is solvable for $n=2$ by Lemma \ref{le12agostoA}.
\begin{lemma}
\label{lemagosto10A} Let $N$ be a normal subgroup of $\aut(\cX)$ such that the quotient curve $\bar{\cX}=\cX/N$ is rational. Then $N$ is not isomorphic to $PSU(3,q)$ with $q\equiv 1 \pmod 4$.
\end{lemma}
\begin{proof}  Let $\mu=3$ or $\mu=1$ according as $3$ divides $q+1$ or does not, and factorize the order of $PSU(3,q)$ as $q^3(q^2-q+1)(q-1)(q+1)^2/\mu$.

Assume first that $p$ is prime to $q$. Since a Sylow subgroup $M$ of $PSU(3,q)$ of order $q^3$ has exponent at most $q$, Lemma \ref{lemagosto10} applied to $M$ yields $q^2\mid(p-2)$. On the other hand, as $p$ divides one of the integers $q^2-q+1,q-1,q+1$, we have $p<q^2$. This contradiction proves the claim for $(p,q)=1$.

Assume that $q=p^m$ for some $m\geq 1$. Take a subgroup in $PSU(3,q)$ that is the direct product of two cyclic groups $C$ and $C_1$ both of odd order $\textstyle\frac{1}{2}(q+1)/\mu$. Write $|C|=p_1^{u_1} \cdots p_t^{u_t}$ with $p_1,\ldots, p_t$ pairwise distinct prime numbers. Obviously, the subgroup $G_i$ of $G$ of order $p_i^{2u_i}$ has exponent $p^{u_i}$. Since $p\nmid(q+1)/\mu$, Lemma \ref{lemagosto10} applied to $G_i$ yields that $p_i^{u_i}$ divides $p-2$. Therefore, $|C|$ itself divides $p-2$ showing that $(\textstyle\frac{1}{2}(q+1)/\mu)\mid(p-2)$. From this, $\lambda(p^m+1)=2\mu(p-2)$ for a positive integer $\lambda$, whence $p^m\in \{5,17\}$ follows.
We may assume that $S$ contains $M$.

We show that $S=M$. For $q\in \{5,17\}$,  $|\aut(PSU(3,q))|=6|PSU(3,q)|$ holds, and hence no element in $S\setminus M$ is in $\aut(PSU(3,q))$. Therefore, if we suppose $S$ to be larger than $M$, the elements of $S$  not in $M$ commute with $M$. According to (v) of Lemma \ref{30ag2013}, take a pair $\{s_1,s_2\}$ of generators of $S$, both of order $p$. Obviously, one of them, say $s_1$, is not in $M$. Then $s_2$ is not $M$ as well, otherwise $|S|=p^2<p^3=|M|$. Therefore, every element in $M$ is falls in $Z(S)$ as both $s_1$ and $s_2$ commute with $M$. But then $M$ is contained in $Z(S)$ which is impossible since $M$ is not abelian.

It remains to rule out the possibility that either $|S|=|M|=5^3$ or $|S|=|M|=17^3$. Assume first that $|S|=5^3$. From Propositions \ref{propgp=3} and \ref{a30ag2013}, the quotient curve $\bar{\cX}=\cX/Z(S)$ is a Nakajima extremal curve of genus $\bar{\gg}=(p-2)p=15$. By Lemma \ref{le12agostoA},
a Sylow $2$-subgroup of $\aut(\cX)$ is a subgroup of a dihedral group of order $2(p-1)=8$. On the other hand, the normalizer $T$ of $Z(S)$ in $PSU(3,5)$ has order $1000=8\cdot 125$ and its factor group $\bar{T}=T/Z(S)$  has a cyclic group of order $8$. Since $\bar{T}$ is a subgroup of $\aut(\bar{\cX})$, this is impossible. The proof for $|S|=17^3$ is analogous. In fact, the normalizer $T$ of $Z(S)$ in $PSU(3,17)$ has order $32\cdot 3\cdot 17^3$ and the factor group
$\bar{T}=T/Z(S)$  has a cyclic group of order $32$.
 \end{proof}
\begin{lemma}
\label{lemagosto12B} Let $N$ be a normal subgroup of $\aut(\cX)$ such that the quotient curve $\bar{\cX}=\cX/N$ is rational. Then $N$ is not isomorphic to $PSL(3,q)$ with $q\equiv 3 \pmod 4$.
\end{lemma}
\begin{proof} We argue as in the proof of Lemma \ref{lemagosto10A}. Let $\mu=3$ or $\mu=1$ according as $3$ divides $q-1$ or does not, and factorize the order of $PSL(3,q)$ as $q^3(q^2+q+1)(q+1)(q-1)^2/\mu$.

Assume first that $p$ is prime to $q$. Since a Sylow subgroup $M$ of $PSL(3,q)$ of order $q^3$ has exponent at most $q$, Lemma \ref{lemagosto10} applied to $M$ yields $q^2\mid(p-2)$. On the other hand, as $p$ divides one of the integers $q^2+q+1,q-1,q+1$, we have either $p<q^2$, or $p=q^2+q+1$. Both cases are inconsistent with $q^2\mid (p-2)$. This contradiction proves the claim for $(p,q)=1$.

Assume that $q=p^m$ for some $m\geq 1$. Then $p\equiv 3 \pmod 4$. Take a subgroup in $PSL(3,q)$ that is the direct product of two cyclic groups $C$ and $C_1$ both of odd order $\textstyle\frac{1}{2}(q-1)/\mu$. Write $|C|=p_1^{u_1} \cdots p_t^{u_t}$ with $p_1,\ldots, p_t$ pairwise distinct prime numbers. Obviously, the subgroup $G_i$ of $G$ of order $p_i^{2u_i}$ has exponent $p^{u_i}$. Since $p\nmid(q-1)/\mu$, Lemma \ref{lemagosto10} applied to $G_i$ yields that $p_i^{u_i}$ divides $p-2$. Therefore, $|C|$ itself divides $p-2$ showing that $(\textstyle\frac{1}{2}(q-1)/\mu)\mid(p-2)$. From this, $\lambda(p^m-1)=2\mu(p-2)$ for a positive integer $\lambda$, whence either $p^m=3$, or
$p^m=7$ follow. We may assume that $S$ contains $M$. As in the proof of Lemma \ref{lemagosto10A}, this implies  $S=M$ since $|\aut(PSL(3,3))|=2|PSL(3,3)|$ and $|\aut(PSL(3,7))|=6|PSL(3,7)|$.

Assume that $p^m=7$. Then $N\cong PSL(3,7)$, and $S\cong UT(3,7)$ whose center $Z(S)$ has order $7$. The normalizer $L$ of $Z(S)$ in $N$ has order $4116=7^3\cdot 12$, and the factor group $L/Z(S)$ is the semidirect product of a normal subgroup $S/Z(S)$ of order $7^2$ by an abelian subgroup of order $12$. Such a group $L/Z(S)$ is a subgroup of the $\mathbb{K}$-automorphism group of the Nakajima extremal curve $\cX/Z(S)$ of genus $15=7\cdot (7-5)+1$.
Since a dihedral group of order bigger than $4$ is not abelian, this contradicts Lemma \ref{le12agostoA}.


Assume that $p^m=3$. Take a subgroup $C$ of $PSL(3,3)$ of order $13$. The Hurwitz formula applied to $C$ yields that $9=13(\bar{\gg}-1)+6\lambda$ where $\bar{\gg}$ is the genus of the quotient curve $\bar{\cX}=\cX/C$ and $\lambda$ is an integer. Therefore, $\bar{\gg}=0$ and hence $22=6\lambda$ which is impossible.
\end{proof}
\begin{lemma}
\label{lemagosto10B} Let $N$ be a normal subgroup of $\aut(\cX)$ such that the quotient curve $\bar{\cX}=\cX/N$ is rational. Then $N$ is not isomorphic to $PSL(2,q)$ with $q\geq 5$.
\end{lemma}
\begin{proof} Assume on the contrary that $N\cong PSL(2,q)$ with $q\ge 5$, and choose a Sylow $p$-subgroup $T$ of $N$. By Lemma \ref{lemagosto9D}, $T$ is a subgroup of $S$ of index at most $p$. By Proposition \ref{31ago2013}, $T$ is a non-cyclic group. From  the classification of subgroups of $PSL(2,q)$, see \cite[Chapter II. Hauptsatz 8.27]{huppertI1967} and \cite{maddenevalentini1982}, $T$ is an elementary abelian group of order $q$ where $q$ is a power of $p$. If $S=T$ then $S$ is elementary abelian as well, and hence $|S|=p^2$, by Proposition \ref{propbp=3}. But then, by Lemma \ref{le12agostoA}, $\aut(\cX)$ is solvable and hence contains no subgroup isomorphic to $PSL(2,q)$ with $q\geq 5$.

Therefore,  $[S:T]=p$. We show that $q=p^r$ with $r$ divisible by $p$.  Take an element $s\in S$ not in $T$. Since $s$ normalizes $N$, either $s$ induces an automorphism of $N$, or centralizes $N$. The latter case  cannot actually occur as $S$ is not abelian by Proposition \ref{propbp=3}. Thus $s\in \aut(N)$. From \cite[Chapter II, Aufgabe 15]{huppertI1967}, the automorphism group of $PSL(2,p^r)$ is $P\Gamma L(2,p^r)$. Since $P\Gamma L(2,p^r)$ only contains $p$-elements other than those in $PSL(2,p^r)$ when $p\mid r$, we have that $r=\lambda p$ for an integer $\lambda$.

The normalizer of $T$ in $N$ is a semidirect product $T\rtimes C$ with a cyclic group $C$ of order $\textstyle\frac{1}{2}(q-1)$. Since $T$ is a normal subgroup of $S$, the normalizer of $T$ in $\aut(\cX)$ also contains $S$.
Actually, $S$ also normalizes $T\rtimes C$. In fact, since $S$ normalizes $T$, any subgroup $s^{-1}(T\rtimes C)s$ with $s\in S$ is a subgroup of $N$ containing $T$. Since $p\geq 5$, the classification of subgroups of $PSL(2,q)$, see \cite[Chapter II. Hauptsatz 8.27]{huppertI1967} and \cite{maddenevalentini1982}, yields that $N$ has a unique subgroup of order $\textstyle\frac{1}{2}q(q-1)$ containing $T$. Therefore, $s^{-1}(T\rtimes C)s=T\rtimes C$.
It turns out that $S(T\rtimes C)$ is a subgroup of the normalizer of $T$ in $\aut(\cX)$ whose order is $\textstyle\frac{1}{2}(q-1)|S|$. Therefore, since $[S:T]=p$, the factor group $S(T\rtimes C)/T$ has order $\textstyle\frac{1}{2}p(q-1)$, and it may be regarded as a $\mathbb{K}$-automorphism group of the quotient curve ${\cY}=\cX/T$.
Observe that $\textstyle\frac{1}{2}p(q-1)\geq \ha5(5^5-1)>60$.

Two cases arise according as ${\cY}$ is rational or not.

In the former case, $S(T\rtimes C)/T$ is isomorphic to a subgroup of $PGL(2,\K)$. From the classification of subgroups of $PSL(2,\mathbb{K})$, see \cite[Chapter II. Hauptsatz 8.27]{huppertI1967} and \cite{maddenevalentini1982},  $q=p$ must hold. But we have already shown that $r>1$, a contradiction.

In the latter case, Proposition \ref{pro27luglio2014} yields that $T$ is one of the subgroups $M_i$ with $3\le i \le p+1$, and hence  by Proposition \ref{pro18feb2014} the curve ${\cY}$ satisfies the hypotheses of Proposition \ref{igusa1}.
For $p>3$, Proposition \ref{igusa1} yields that $C$ is isomorphic to a subgroup of a dihedral group of order $2(p-1)$. Therefore $\textstyle\frac{1}{2}(q-1)$ divides $p-1$. Since $q=p^r$ with $r>1$ is this is impossible.
For $p=3$, Proposition \ref{igusa1} gives some more possibilities namely that $C$ is isomorphic to a cyclic subgroup of $GL(2,3)$. Then $|C|\in \{2,3,4,6,8\}$, but none of these number is equal to $\textstyle\frac{1}{2}(q-1)$ for $q=3^{r}$ with $r$ divisible by $3$.
\end{proof}
\begin{lemma}
\label{lemagosto12E} Let $N$ be a normal subgroup of $\aut(\cX)$ such that the quotient curve $\bar{\cX}=\cX/N$ is rational. Then $N$ is not isomorphic to  $N\cong {\rm{Alt}_7}$ or $N\cong {\rm{M}}_{11}$.
\end{lemma}
\begin{proof} Since both $\rm{Alt}_7$ and  $\rm{M}_{11}$ have subgroups of odd non-prime order $d$ only for $d=9$, Lemma \ref{lemagosto9D} yields $p=3$ and $n=3$. Since the quotient curve $\bar{\cX}=\cX/N$ is rational, and  neither $\rm{Alt}_7$ nor  $\rm{M}_{11}$ has an outer automorphism of order $3$, the case $n=3$ can only occur if each element of $S\setminus N$ centralizes $N$. But then $S$ would be abelian contradicting Proposition \ref{propbp=3}.
\end{proof}
\begin{proposition}
\label{pro9agostoA} Let $N$ be a minimal normal subgroup of $\aut(\cX)$ such that the quotient curve $\bar{\cX}=\cX/N$ is rational. Then  $N$ is an elementary abelian group.
\end{proposition}
\begin{proof} Assume on the contrary that $N$ is isomorphic to the direct product $R_1\times \ldots\times R_k$ of pairwise isomorphic non-abelian simple groups. Let $U_i$ be a Sylow $2$-subgroup of $R_i$ for $i=1,\ldots k$. By Remark \ref{2sub}, $U_i$ is either dihedral or semidihedral. Therefore $N$ contains a $2$-subgroup which is the direct product of $k$ dihedral, or semidihedral groups. This implies for $k>1$ that $N$ contains an elementary abelian subgroup of order $8$, but this contradicts Lemma \ref{lem6agos}. Therefore $k=1$. Now, the assertion follows from Remark \ref{2sub} together with Lemmas \ref{lemagosto10A}, \ref{lemagosto12B},  \ref{lemagosto10B}, and \ref{lemagosto12E}.
\end{proof}
\begin{lemma}
\label{6ago2014} Let $U$ be a $2$-subgroup of $\aut(\cX)$. If $U$ normalizes $M_1$ (or $M_2$) then $U$ is cyclic.
\end{lemma}
\begin{proof} By Proposition \ref{pro27luglio2014}, $M_1$ has $p$ orbits on $\Omega_1$ each of length $p^{n-2}$. Since $\Omega_1$ is the set of points which are fixed by some non-trivial elements of $M_1$, $U$ preserves $\Omega_1$, and induces a permutation group on the set of the $p^{n-2}$ $M_1$-orbits. As $U$ has order a power of $2$, it preserves some of these $M_1$-orbits. Since the length
of such a $U$-invariant $M_1$-orbit is odd, some point of it must be fixed by $U$. Therefore, $U$ fixes a point of $\cX$, and hence $U$ is cyclic.
\end{proof}

We are in a position to prove Theorem \ref{fullaut}.

Our proof is by induction on the order of $S$. The assertion holds for $|S|=p^2$ by Lemma \ref{le12agostoA}. Assume that it holds for all extremal Nakajima curves with Sylow $p$-subgroup of order $p^k$ with $2\le k \le n-1$.  Take a minimal normal subgroup $N$ of $\aut(\cX)$. If the quotient curve $\bar{\cX}=\cX/N$ is not elliptic then Lemmas \ref{lemagosto9C} and \ref{lemagosto9D}  together with Proposition \ref{pro9agostoA} show that $N$ is a $p$-group and hence it is a subgroup of $S$. If $\bar{\cX}=\cX/N$ is elliptic and $N$ is not a $p$-group,
replace $N$ with $\Phi(S)$ when $S$ is a normal subgroup of $\aut(\cX)$, otherwise replace $N$ with
or $M_1$ (or $M_2$) according to Lemma \ref{lemagosto9E}. Therefore, $N$ may be assumed to be a $p$-group.

If $N$ is semiregular on $\cX$, then the quotient curve $\bar{\cX}=\cX/N$ has positive $p$-rank, and one of the cases (ii) or (iii) of Theorem \ref{princ} occurs. Therefore, $\bar{\cX}$ is either an extremal Nakajima curve, or a curve of genus $p-1$ given in Proposition \ref{igusa1}, where $S/N$ is a Sylow $p$-subgroup of $\aut(\bar{\cX})$. In case (iii), Theorem \ref{fullaut} holds for $\bar{\cX}$ by induction, and accordingly let $\bar{L}=\bar{S}$ when $\bar{S}$ is a normal subgroup of $\aut(\bar{\cX})$, but let $\bar{L}=\bar{M}$ when the sporadic case $p=3$ with  $GL(2,3)$ occurs. In case (ii), Proposition \ref{igusa1} holds for $\bar{\cX}$, and
let $\bar{L}=\bar{S}$ when $\bar{S}$ is a normal subgroup of $\aut(\bar{\cX})$, but let $\bar{L}$ be the identity subgroup when the sporadic case $p=3$ with $GL(2,3)$ occurs.
Since $\bar{L}$ is contained in $S/N$, there exists a normal subgroup $L$ of $\aut(\cX)$ containing $N$ such that $L/N=\bar{L}$. Then $L$ is a $p$-group and 
$$\frac{\aut(\cX)}{L}\cong \frac{\aut(\cX)/N}{L/N}\cong \frac{\bar{G}}{\bar{L}}$$ 
where $\bar{G}$ is a subgroup of $\aut(\bar{\cX})$. 
If $\bar{S}=\bar{L}$ then $S=L$ and hence $\bar{G}$ has order prime to $p$. By induction, $\bar{G}$ is a subgroup of a dihedral group of order $2(p-1)$, and hence Theorem \ref{fullaut} holds.
If $[\bar{S}:\bar{L}]=p$ then $p=3$, and  $3\mid |G|$. By induction, $\bar{G}$ is isomorphic to a subgroup of $GL(2,3)$, and hence Theorem \ref{fullaut} holds.

If $N$ is not semiregular on $\cX$, Proposition \ref{propcp=3} shows that $N=M_1$ (or $N=M_2$). From Proposition \ref{pro27luglio2014}, the quotient curve $\bar{\cX}=\cX/N$ is rational. Therefore,
$\aut(\cX)/N$ is isomorphic to a subgroup $\Gamma$ of $PGL(2,\mathbb{K})$. As $S$ is a Sylow $p$-subgroup of $\aut(\cX)$ containing $M_1$ and $[S:M_1]=p$, the order of $\Gamma$ is divisible by $p$ but not by $p^2$. Also, a Sylow $2$-subgroup of $\Gamma$ is cyclic, by Lemma \ref{6ago2014}. In particular, $\Gamma$ is not isomorphic to $\Alt_4$, or $\Sym_4$, or $\Alt_5$, or $PSL(2,q)$, or $PGL(2,q)$ with a power $q$ of $p$. From the classification of finite subgroups of $PGL(2,\mathbb{K})$, see \cite{maddenevalentini1982} or \cite[Theorem A.8]{hirschfeld-korchmaros-torres2008}, we are left with only one possibility for $\Gamma$, namely a subgroup of the semidirect product of $S/M_1$ by a cyclic group whose order divides $p-1$.
Hence Theorem \ref{fullaut} holds.

Our proof of Theorem \ref{fullaut} also shows that if $\K(\cX)$ is not an unramified Galois extension of the Artin-Mumford function field then the dihedral subgroup of order $2(p-1)$ may be weakened to the cyclic group of order $p-1$.

\section{Nakajima extremal curves with small genera for $p=3$}
\begin{proposition} Let $p=3$. If $S$ has maximal class then
$\Phi(S)$ is an abelian metacyclic group.
\end{proposition}
\begin{proof} 
We may assume that $|\Phi(S)|=3^m$ with $m\geq 3$. From (ii) of Proposition \ref{30ag2013}, $|S|=3^{m+2}\geq 3^5$. From \cite[Theorem 5.2]{berkovich2002}, every subgroup of $S$ can be generated by two elements. Therefore, $d(\Phi(S))=2$. Assume on the contrary that $\Phi(S)$ is not abelian. From \cite[Theorem 3]{langegail}, $\Phi(S)$ is metacyclic. Since $\Phi(S)$ is supposed to be non-abelian, \cite[Theorem 1]{langegail} shows the existence of a metacyclic subgroup $B$ of $S$ such that $\Phi(B)=\Phi(S)$. By Proposition \ref{d4set2013}, $B$ is a proper subgroup of $S$ containing $\Phi(S)$. Since $B$ is finite, $B\neq \Phi(B)$ and hence $[B:\Phi(S)]=p$. The Burnside fundamental theorem, \cite[Chapter III, Satz 3.15]{huppertI1967} yields that $B$ and hence $\Phi(S)$ is cyclic, a contradiction.
\end{proof}
\subsection{Cases $|S|=3,9$}
We prove that if $\cX$ satisfies the hypotheses of Theorem \ref{princ} for $|S|=3$ then (ii) holds.
For this case, our hypothesis (\ref{hyp}) yields $\gg=2$. From (\ref{eq2deuring}), every automorphism of $\aut(\cX)$ of order $3$ has two fixed points on $\cX$. Therefore, (i) of Theorem \ref{princ} cannot occur, and  the assertion follows from Proposition \ref{igusa1}.

{}From now on, $|S|=9$ and  $\cX$ is a curve satisfying the hypotheses of Theorem \ref{princ} but does not have the property given in (i) of Theorem \ref{princ}
\begin{proposition}
\label{pro28luglio2014} Let $p=3$. Up to isomorphisms, the Artin-Mumford curve with affine equation (\ref{artinmumford}) is the unique extremal Nakajima curve of genus $4$.
\end{proposition}
\begin{proof}
{}From Propositions  \ref{propap=3} and \ref{propbp=3}, $\cX$ is an ordinary curve of genus $\gg=4$ with  an elementary abelian subgroup $S$ of $\aut(\cX)$ of order $9$.

Let $N$ be the kernel of the permutation representation of $S$ on $\Omega_1\cup \Omega_2$. If $N$ is not trivial then it has order $3$, and the Hurwitz genus formula applied to $N$ gives $6=2(\gg-1)\ge 6(\bar{\gg}-1)+24$.
Therefore $S$ acts on $\Omega_1\cup\Omega_2$ faithfully.

By Proposition \ref{prop6febb2013}, $\cX$ is assumed to be a canonical curve embedded in $PG(3,\mathbb{K})$. Then $S$ extends to a subgroup of $PG(3,\mathbb{K})$
which preserves $\cX$ and acts on $\cX$ faithfully.

According to Lemma \ref{flag}, choose the projective coordinate system $(X_0:X_1:X_2:X_3)$ in $PG(3,\mathbb{K})$ in such a way that 
$S$ preserves the canonical flag $$P_0\subset\Pi_1\subset \Pi_2$$
where $P_0=(1:0:0:0)$, $\Pi_1$ is the line through $P_0$ and $P_1=(0:1:0:0)$ while $\Pi_2$ is the plane of equation $X_3=0$. Here $P_0\not\in \cX$, since $S$ fixes no point in $\cX$. Moreover,
$\Pi_2\cap\cX=\Omega_1\cup \Omega_2$. In fact, for any point $R\in \Pi_2\cap \cX$,  Proposition \ref{propap=3} implies that the $S$-orbit of $R$ has size $9$ unless $R\in \Omega_1\cup \Omega_2$. On the other hand $S$ preserves $\Pi_2\cap \cX$, and this implies that the $S$-orbit of $R$ cannot exceed $6$.
\begin{lemma}
\label{lem7feba2013} Both $\Omega_1$ and $\Omega_2$ consist of three collinear points.
\end{lemma}
\begin{proof} Assume on the contrary that $\Omega_1$ is a triangle. Take $g\in S$ such that $g$ fixes each vertex of $\Omega_1.$ Since $g$ is a projectivity of $PG(3,\mathbb{K})$ it fixes $\Pi_2$ pointwise. As $\Pi_2$ also contains  $\Omega_2,$ $g$ must fix $\Omega_2$ pointwise. But this is impossible as $S$ acts on $\cX$ faithfully.
\end{proof}
As a corollary, $I(R,\cX\cap \Pi_2)=1$ for every point $R\in\Omega_1\cup \Omega_2$.  For $i=1,2$, let $r_i$ denote the line containing $\Omega_i$. Their common point is fixed by $S$, and may be chosen for $P_0$. Let $M$ be the subgroup of $S$ which preserves every line through $P_0$. Since $\deg\, \cX=6$, no line meets $\cX$ in more than six distinct points. Therefore, either $|M|=1$ or $|M|=3$. In the latter case, $M$ is an elation group of order $3$ with center $P_0$. If $\Delta$ is its axis then
every point in $\Delta\cap \cX$ is fixed by $M$. Therefore $\Delta$ is not $\Pi_2$ and contains either $r_1$ or $r_2$. Since $S$ is abelian, it  preserves $\Delta$ and hence every plane through $r_1$. But then every $S$-orbit has length at most $3$. A contradiction with Proposition \ref{propap=3}. Hence $M$ is trivial.

Since $P_0\not \in \cX$, the linear system $\Sigma$ of all planes through $P_0$ cuts out on $\cX$ a linear series without fixed point. Therefore this effective linear series has dimension $2$ and degree $6$, and is denoted by $g_2^6$.

It might happen that $g_2^6$ is composed of an involution, and we investigate such a possibility. From \cite[Section 7.4]{hirschfeld-korchmaros-torres2008}, there is a curve ${\cZ}$ whose function field $\mathbb{K}({\cX})$ is an $S$-invariant proper subfield of $\mathbb{K}(\cX)$. Since no non-trivial element in $S$ fixes every line through $P_0$ and hence every plane through $P_0$, $S$ acts on ${\cZ}$ faithfully.  As the genus of $\cZ $ is less than $4$,  applying \eqref{naka16feb2013} to $\cZ$ gives   $\gamma(\cZ)=0$. Therefore, every non-trivial element in $S$ has a unique fixed point $\bar{T}$, see  \cite[Lemma 11.129]{hirschfeld-korchmaros-torres2008}.
{}From this, the support of the divisor of $K(\cX)$ lying over $\bar{T}$ contains the points in $\Omega_1\cup \Omega_2$. Therefore, the line through $P_0$ and a point in $\Omega_1\cup \Omega_2$ must contain all the points in $\Omega_1\cup \Omega_2$. But this would imply that $r_1=r_2$, a contradiction.

Therefore, $g_2^6$ is simple and without fixed point. The projection of $\cX$ from $P_0$ is an irreducible plane curve $\cC$ of degree $6$ and genus $4$ with two triple points $R_1$ and $R_2$ arising from $\Omega_1$ and $\Omega_2$, respectively. Here $\cC$ and $\cX$ are birationally equivalent, and $S$ is a subgroup of $PGL(3,\mathbb{K})$ preserving $\cC$.  For $i=1,2$, a non-trivial projectivity $s_i\in S$ fixing $\Omega_i$ pointwise acts on $\cC$ fixing the point $R_i$.

Choose the projective coordinate system $(X_0:X_1:X_2)$ in $PG(2,\mathbb{K})$ so that $R_1=(0:0:1)$ and $R_2=(0:1:0)$. In affine coordinates $(X,Y)$ with $X=X_1/X_0,\,Y=X_2/X_0$, an equation of $\cC$ is  $f=0$ with an irreducible polynomial $f\in \mathbb{K}[X,Y]$ of degree six. W.l.o.g. the origin $O=(0,0)$ is the common point of two tangents to $\cC$, say $t_1$ at $R_1$ and $t_2$ at $R_2$. Furthermore,  $s_1(O)=(\lambda,0)$, $s_2(O)=(0,\mu)$ with $\lambda,\mu\in \mathbb{K}^*$, and $\lambda=\mu=1$ may be assumed.
Thus $s_1:\,(X,Y)\mapsto (X+1,Y)$ and $s_2:\,(X,Y)\mapsto (X,Y+1)$. Hence
\begin{equation}
\label{eq7feb2013} f(X+1,Y)=f(X,Y),\quad f(X,Y+1)=f(X,Y).
\end{equation}
Since $R_1$ is a triple point of $\cC$, there exist $h_0,h_1,h_2,h_3\in \mathbb{K}[Y]$ such that $$f(X,Y)=h_3X^3+ h_2X^2+h_1X+h_0=0,$$
where $\deg\, h_0\leq 2$ by the particular choice of $t_2$.
{}From this and (\ref{eq7feb2013}), the polynomial
\begin{equation}
\label{eq7febb2013}
f(X+1,Y)-f(X,Y)=h_3-h_2X+h_2+h_1
\end{equation}
 vanishes at every affine point of $\cC$. Since $\cC$ is not rational, this is only possible when (\ref{eq7febb2013}) is the zero polynomial, that is,   $h_2=
h_3+h_1=0.$ Thus $f(X,Y)=h_3(X^3-X)+h_0$. The second mixed partial derivate is $f_{X,Y}=-dh_3/dY$. Similarly, as $R_2$ is a triple point of $\cC$ there exist $k_0,k_3\in \mathbb{K}[X]$ with $\deg\,k_0\leq 2$ such that $f(Y,X)=k_3(Y^3-Y)+k_0$. Since $f_{X,Y}=f_{Y,X}$, this yields $dh_3/dY=dk_3/dX$,  whence $dh_3/dY$ and $dk_3/dX$  both have degree $0$. Thus
$$h_3=c_3Y^3+c_1Y+c_0, \quad   k_3=d_3X^3+d_1X+d_0$$
where $c_0,c_1,c_3,d_0,d_1,d_3\in \mathbb{K}.$ Therefore
$$(c_3Y^3+c_1Y+c_0)(X^3-X)+h_0=(d_3X^3+d_1X+d_0)(Y^3-Y)+k_0.$$
Comparison of the coefficients of $X^3$  shows that $c_3=-c_1,c_0=0.$ Similarly, $d_3=-d_1,d_0=0.$ Thus
$$f(X,Y)=c(X^3-X)(Y^3-Y)+h_0=d(Y^3-Y)(X^3-X)+k_0$$
where $c,d\in \mathbb{K}.$ From this
$(c-d)(X^3-X)(Y^3-Y)=k_0-h_0$ whence $c=d$ and $k_0=h_0=u$ with $u\in \mathbb{K}^*$. Therefore
$$f(X,Y)=(X^3-X)(Y^3-Y)+c=0$$
where $c\in \mathbb{K}^*$.
\end{proof}
\subsection{Case $|S|=27$}
\label{S27}
In this case, the maximal subgroups of $S$ are elementary abelian groups of order $9$ and  Theorem \ref{14mag2014a} applies. Therefore, the Nakajima extremal curves of
genus $10$ are the curves $\cX_c$ as given in Proposition \ref{14may2014d}. A different presentation of the function field $\mathbb{K}(\cX_c)$ of $\cX_c$ is $\mathbb{K}(\cX_c)=\mathbb{K}(u,v,y,x)$ where
\begin{itemize}
\item[\rm(i)] $u(v^3-v)+u^2-c=0$;
\item[\rm(ii)] $y^3-y-u=0$;
\item[\rm(iii)] $(z^3-z)(v^3+1)+v^3-v^2-u=0.$
\end{itemize}
Here, both $\mathbb{K}(u,v,y)$ and $\mathbb{K}(u,v,z)$ are unramified degree $p$ Galois-extensions of $\mathbb{K}(u,v)$, and $\mathbb{K}(\cX_c)$ can be obtained as the special case $p=3,N=1$ of the construction given in Section \ref{inffam}.
 \subsection{Case $|S|=81$}
\begin{lemma}
\label{prop20bapr2013}
For $|S|=81$ there are only two possibilities for $S$, namely
\begin{itemize}
\item[{\rm{(a)}}] $S\cong S(81,7)$ where $S(81,7)=C_3 \wr C_3$ is the Sylow $3$-subgroup  of the symmetric group of degree $9$, moreover $M_1\cong C_3\times C_3 \times C_3$, $M_2\cong UT(3,3)$, $M_3\cong M_4\cong C_9 \rtimes C_3$.
\item[{\rm{(b)}}] $S\cong S(81,9)=\langle a,b,c |a^9=b^3=c^3=1,ab=ba,cac^{-1}=ab^{-1},cbc^{-1}=a^3b\rangle$ with exactly $62$ elements of order $3$; moreover $M_1\cong M_2\cong M_3 \cong UT(3,3)$, $M_4\cong C_9\times C_3$.
\end{itemize}
\end{lemma}
\begin{proof} There exist exactly seven groups of order $81$ generated by two elements, namely $S(81,i)$ with $i=1,\ldots,7$, and each of them has an abelian normal subgroup of index $3$. By Proposition \ref{a4set2013}, $S$ is of maximal class. There are four pairwise non-isomorphic groups of order $81$ and maximal class, namely (a), (b) and
\begin{itemize}
\item[(c)] $S(81,8)\cong\langle a,b,c |a^9=b^3=c^3=1,ab=ba,cac^{-1}=ab,cbc^{-1}=a^3b\rangle$ with $26$ elements of order $3$;
\item[(d)] $S(81,10)\cong (C_9\rtimes C_3)\rtimes C_3$  with  $8$ elements of order $3$.
 \end{itemize}
One of the four maximal normal subgroups of $S(81,8)$ is isomorphic to $U(3,3)$ and hence it contains all elements of order $3$. On the other hand, (iv) of Proposition \ref{30ag2013} yields that
two of the maximal normal subgroups of $S$, namely $M_1$ and $M_2$, have non-trivial $1$-point stabilizer in $\Omega_1$ and $\Omega_2$, respectively. Hence, both must have an element of order $3$ not contained in $\Phi(S)$. Since $M_1\cap M_2=\Phi(S)$, these elements are not in the same maximal normal subgroup. This contradiction shows that (c) cannot actually occur in our situation. Regarding $S(81,10)$, all elements of order $3$ lie in $\Phi(S)$ as $\Phi(S)$ is an elementary abelian group of order $9$. But this is impossible in our situation since $M_1$ must have an element of order $3$ not in $\Phi(S)$   by Propositions \ref{30ag2013} and \ref{b30ago2013}.
\end{proof}
We point out that both cases in Lemma \ref{prop20bapr2013} occur. The curve $\cX$ with function field $\mathbb{K}(x,y,u,s,w)$ defined by the equations
\label{1jun2014e}
\begin{itemize}
\item[\rm(i)] $x(y^3-y)-x^2-1=0$;
\item[\rm(ii)] $u^3-u-x=0$;
\item[\rm(iii)] $(u-y)(w^3-w)-1=0$;
\item[\rm(iv)] $(u-(y+1))(s^3-s)-1=0$.
\end{itemize}
has genus $\gg(\cX)=28$ and it has a $\mathbb{K}$-automorphism group $S\cong S(81,7)$ generated by $g_1,g_2,g_3,g_4,g_5$ where
$$
\begin{array}{llll}
&g_1:\,(x,y,u,w,s)\mapsto (x,y+1,u,s,u-w-s), &  g_2:\,(x,y,u,w,s)\mapsto (x,y+1,u,s,u-w-s),  \\
&g_3:\,(x,y,u,w,s)\mapsto (x,y+1,u+1,w,s), &  g_4:\,(x,y,u,w,s)\mapsto(x,y,u,w+1,s), \\
&g_5:\,(x,y,u,w,s)\mapsto(x,y,u,w,s+1). &  {}
\end{array}
$$
To show an example for the other case, we apply Theorem \ref{31mag2014} for $N=2$ and obtain a a Nakajima extremal curve of genus $82$ with a $\mathbb{K}$-automorphism group $S$
such that
\begin{itemize}
\item[\rm(i)] $S$ is isomorphic to the unique group $S(243,26)$ of order $243$ with $170$ elements of order $3$, moreover $M_2\cong M_3\cong M_4\cong S(81,9)$, and $M_1\cong C_9\times C_9$.
\end{itemize}
Since $|Z(S)|=3$, Proposition \ref{a30ag2013} applied to $N=Z(S)$ yields the existence of a Nakajima extremal curve of genus $28$ with a $\mathbb{K}$-automorphism group isomorphic to $S/Z(S)$.
Here $S/Z(S)\cong S(81,10)$ and therefore this curve provides an example for Case (b).
\subsection{Case $|S|=243,729$}
\begin{proposition}
\label{b1sep2013} If
 $|S|=243$ and $S$ has a maximal abelian subgroup, then there are only two possibilities for $S$, namely $\rm{(i)}$ and
\begin{itemize}
\item[\rm(ii)] $S$ is isomorphic to the unique group $S(243,28)$ of order $243$ with $116$ elements of order $3$, moreover $M_1\cong M_2\cong S(81,9)$ while $M_3\cong S(81,4)$, and $M_4\cong S(81,10)$.
\end{itemize}
\end{proposition}
\begin{proof} There exist exactly six pairwise non-isomorphic groups of order $81$ and maximal class, namely (i), (ii) and $S(243,25)$ with $62$ elements of order $3$; $S(243,27)$ with $8$ elements of order $3$; $S(243,29)$ with $8$ elements of order $3$; $S(243,30)$ with $62$ elements of order $3$.

One of the four maximal normal subgroups of $S(243,28)$ (and of $S(243,30)$) is isomorphic to $S(81,8)$ and hence it contains all elements of order $3$. The argument in the proof of Proposition \ref{prop20bapr2013} ruling out possibility (c) also works in this case. Therefore, neither $S\cong S(243,25)$ nor $S\cong S(243,28)$ is possible. Regarding $S(243,27)$ and $S(243,29)$, we may use the argument from the proof of Proposition \ref{prop20bapr2013} that ruled out possibility (d). Therefore, $S\cong S(243,25)$ and $S\cong S(243,28)$ cannot occur in our situation.
\end{proof}
Theorem \ref{20luglio2014} applied to $p=3,\,N=1$ provides a Nakajima extremal curve $\cX$ of genus $\gg=244$ and $|S|=729$ so that $\Phi(S)$ is the direct product of two cyclic groups of order $9$. Using this and some other properties of $S$ established before and relying on the database of GAP, it is possible to prove that $S=S(729,34)$. Therefore, $S$ has nilpotency class $4$ and $|Z(S)|=3$. Moreover, $|\aut(\Phi(S))|=2^9\cdot3^5\cdot5\cdot 11$ which is equal to $(3^4-1)(3^4-3)(3^4-3^2)(3^4-3^3)$. Since $d(\Phi(S))=4$, this shows that $\Phi(S)$ hits the Burnside-Hall bound (\ref{eq25agostoA}) and hence $\cX$ is the unique Nakajima extremal curve of genus $\gg=244$ with $S=S(729,34)$. The quotient curve $\bar{\cX}=\cX/Z(S)$ is a Nakajima extremal curve of genus $\gg=82$ and its $\K$-automorphism group $\bar{S}=S/Z(S)$ is $S(243,3)$. In particular, $\bar{S}$ has nilpotency class $3$ and $Z(\bar{S})=9$. Moreover, $Z(\bar{S})$ contains two subgroups, say $\bar{T}_1$ and $\bar{T}_2$, of order $3$ so that the arising quotient curves $\bar{\cX}/\bar{T}_1$ and $\bar{\cX}/\bar{T}_2$ are non-isomorphic Nakajima extremal curves of genus $28$. Therefore, they are the curves given in  Lemma \ref{prop20bapr2013}.

\vspace{0,5cm}\noindent {\em Authors' addresses}:

\vspace{0.2 cm} \noindent Massimo GIULIETTI \\
Dipartimento di Matematica e Informatica
\\ Universit\`a degli Studi di Perugia \\ Via Vanvitelli, 1 \\
06123 Perugia
(Italy).\\
 E--mail: {\tt giuliet@dipmat.unipg.it}

\vspace{0.2cm}\noindent G\'abor KORCHM\'AROS\\ Dipartimento di
Matematica\\ Universit\`a della Basilicata\\ Contrada Macchia
Romana\\ 85100 Potenza (Italy).\\E--mail: {\tt
gabor.korchmaros@unibas.it }

    \end{document}